\documentclass{article}

\usepackage[english]{babel}

\usepackage[letterpaper,top=2cm,bottom=2cm,left=3cm,right=3cm,marginparwidth=1.75cm]{geometry}

\usepackage{amsmath}
\usepackage{amssymb}
\usepackage{graphicx}
\usepackage[colorlinks=true, allcolors=blue]{hyperref}
\usepackage{amsthm}
\usepackage{tikz}
\usepackage{caption}
\usepackage{subcaption}
\usepackage{float}
\usepackage{svg}
\usepackage{mathtools}
\usetikzlibrary{math}
\usetikzlibrary{shapes}

\DeclareMathOperator*{\argmin}{argmin} 
\DeclareMathOperator{\barron}{\mathcal{B}}
\DeclareMathOperator{\Prob}{\mathcal{P}}
\DeclareMathOperator{\F}{\mathcal{F}} 
\DeclareMathOperator{\R}{\mathbb{R}} 
\DeclareMathOperator{\N}{\mathbb{N}} 
\DeclareMathOperator{\E}{\mathcal{E}} 
\DeclareMathOperator{\Fu}{\mathcal{F}}

\newcommand{\comment}[1]{\textcolor{red}{#1}}

\newtheorem{definition}{Definition}
\newtheorem{theorem}{Theorem}
\newtheorem{proposition}{Proposition}
\newtheorem{lemma}{Lemma}
\newtheorem{remark}{Remark}
\newtheorem{corollary}{Corollary}
\newtheorem{problem}{Problem}
\newtheorem{hypothesis}{Hypothesis}

\everymath{\displaystyle}

\title{Numerical solution of Poisson partial differential equation in high dimension using two-layer neural networks}
\author{Dus Mathias, Ehrlacher Virginie}

\begin{document}
\maketitle

\begin{abstract}

The aim of this article is to analyze numerical schemes using two-layer neural networks with infinite width for the resolution of the high-dimensional Poisson partial differential equation (PDE) with Neumann boundary condition. Using Barron's representation of the solution~\cite{Barron1993} with a probability measure defined on the set of parameter values, the energy is minimized thanks to a gradient curve dynamic on the $2$-Wasserstein space of the set of parameter values defining the neural network. Inspired by the work from Bach and Chizat~\cite{BachChizat2018,BachChizat2021}, we prove that if the gradient curve converges, then the represented function is the solution of the elliptic equation considered. In contrast to the works~\cite{BachChizat2018,BachChizat2021}, the activation function we use here is not assumed to be homogeneous to obtain global convergence of the flow. Numerical experiments are given to show the potential of the method.

\end{abstract}

\section{Introduction}\label{sectionIntroduction}

\subsection{Literature review}

The motivation of our work is to bring some contributions on the mathematical understanding of neural-network based numerical schemes, typically Physically-Informed-Neural-Networks (PINNs)~\cite{raissi2019physics, han2018solving, karniadakis2021physics, weinan2021algorithms, cuomo2022scientific,despres2022neural} approaches, for the resolution of some high-dimensional Partial Differential Equations (PDEs). In this context, it is of tremendous importance to understand why neural networks work so well in some contexts in order to improve its efficiency and get an insight of why a particular neural network should be relevant to a specific task. 

The first step towards a mathematical analysis theory of neural network-based numerical methods is the identification of functional spaces suited for neural network approximation. The first important result in this direction is the celebrated theorem of approximation due to Cybenko \cite{Cybenko1989} proving that two-layer neural networks can approximate an arbitrary smooth function on a compact of $\R^d$. However, this work does not give an estimation of the number of neurons needed even if it is of utmost importance to hope for tractable numerical methods. To answer this question, Yarotsky \cite{Yarotsky2017} gave bounds on the number of neurons necessary to represent smooth functions. This theory mainly relies on classical techniques of Taylor expansions and does not give computable architectures in the high dimensional regime. Another original point of view was given by Barron \cite{Barron1993} who used Monte Carlo techniques from Maurey-Jones-Barron to prove that functions belonging to a certain metric space \textit{ie} the Barron space, can be approximated by a two-layer NN with precision $O\left(\frac{1}{\sqrt{m}}\right)$, $m$ being the width of the first layer. Initially, Barron's norm was characterized using Fourier analysis reducing the theory to domain where Fourier decomposition is available. Now other Barron type norms which does not suppose the existence of an harmonic decomposition \cite{WainanMaWu2020}, are also available.

In order to give a global idea of how this works, one can say that a Barron function $f_\mu: \R^d \rightarrow \R$ can be represented by a measure $\mu$ with second order moments :

$$
f_\mu(x) := \int a \sigma(wx + b) d\mu(a,b,c)
$$
where $\sigma$ is an activation function and the Barron norm $\|f_\mu \|_{\mathcal{B}}$ is roughly speaking, a mix of the second order moments of $\mu$. Intuitively, the law of large number says that the function $f_\mu$ can be represented by a sum of Dirac corresponding to a two-layer neural network whose width equals the number of Dirac masses. The architecture of a two-layer neural network is recalled in Figure~\ref{figureDescriptionTwoLayersNN}. Having said that, some important questions arise :

\begin{itemize}
\item What is the size of the Barron space and the influence of the activation function on such size ?

Some works have been done in this direction for the ReLU activation function. In~\cite{WeinanWojtowytsch2022}, it is proven that $H^s$ functions are Barron if $s \geq \frac{d}{2}+2$ and that $f_\mu$ can be decomposed by an infinite sum of $f_{\mu_i}$ whose singularities are located on a $k$ ($k < d$) affine subspace of $\R^d$. For the moment, no similar result seems to hold with more regular activation functions.

\item One can add more and more layers and observe the influence on the corresponding space. In~\cite{WeinanWojtowytsch2020}, tree-like spaces $\mathcal{W}_L$ (where $L$ is the number of hidden layers) are introduced using an iterative scheme starting from the Barron space. Of course, multi-layers neural networks naturally belong to these spaces. Nevertheless for a function belonging to $\mathcal{W}_L$, it is not clear that a multilayer neural network is more efficient than its two-layer counterpart for its approximation.

\item Does solutions of classical PDEs belong to a Barron space ? In this case, there is a potential to solve PDEs without suffering from the curse of dimension. Some important advances have been made in this direction in~\cite{Jianfeng2021} where authors considered the Poisson problem with Neumann boundary condition on the $d$ dimensional cube. If the source term is Barron, then it is proved that the solution is also Barron and there is hope for an approximation with a two-layer NN.  
\end{itemize}

Using conclusions from~\cite{Jianfeng2021}, the object of this paper is to propose and analyze a neural-network based numerical approach for the resolution of the Poisson equation in the high dimensional regime with Barron source. Inspired from~\cite{BachChizat2018}, we immerse the problem on the space of probability measures with finite second order moments defined on the parametric domain. This corresponds to finding a solution to the PDE thanks to an infinitely wide two-layer neural network. Then we interpret the learning phase of the network as a gradient curve in the space of probability measure. Finally under some hypothesis on the initial support, we prove that if the curve converges then it necessarily does towards a measure corresponding to the solution of the PDE considered. Note that our argumentation is different from~\cite{BachChizat2018,BachChizat2021} since the convergence proof is not based on topological degree nor the topological properties of the sphere. We rather use a homology argument taken from algebraic topology and a clever choice of activation function to prove that the dynamic of the support of the gradient curve of measure behaves nicely. Numerical experiments are conducted to confirm the potential of the method proposed.

In Section~\ref{sec:prelim}, the problem is presented in a more precise way and the link between probability and Barron functions is made clearly. In Section~\ref{sectionGradientFlow}, the gradient curve is introduced and our main theorems on its well-posedness and convergence are presented and proved. Finally, numerical experiments are exposed in Section~\ref{sectionNumericalExperiment}. 
\newline

\underline{\textbf{Notation}} : For $1\leq p \leq \infty$, the notation $|\cdot|_p$ designates the $\ell^p$ norm of a vector of arbitrary finite dimension with particular attention to $p=2$ (euclidean norm) for which the notation $|\cdot|$ is preferred.

\tikzmath{\r = 0.9; \D  = 1.5;}

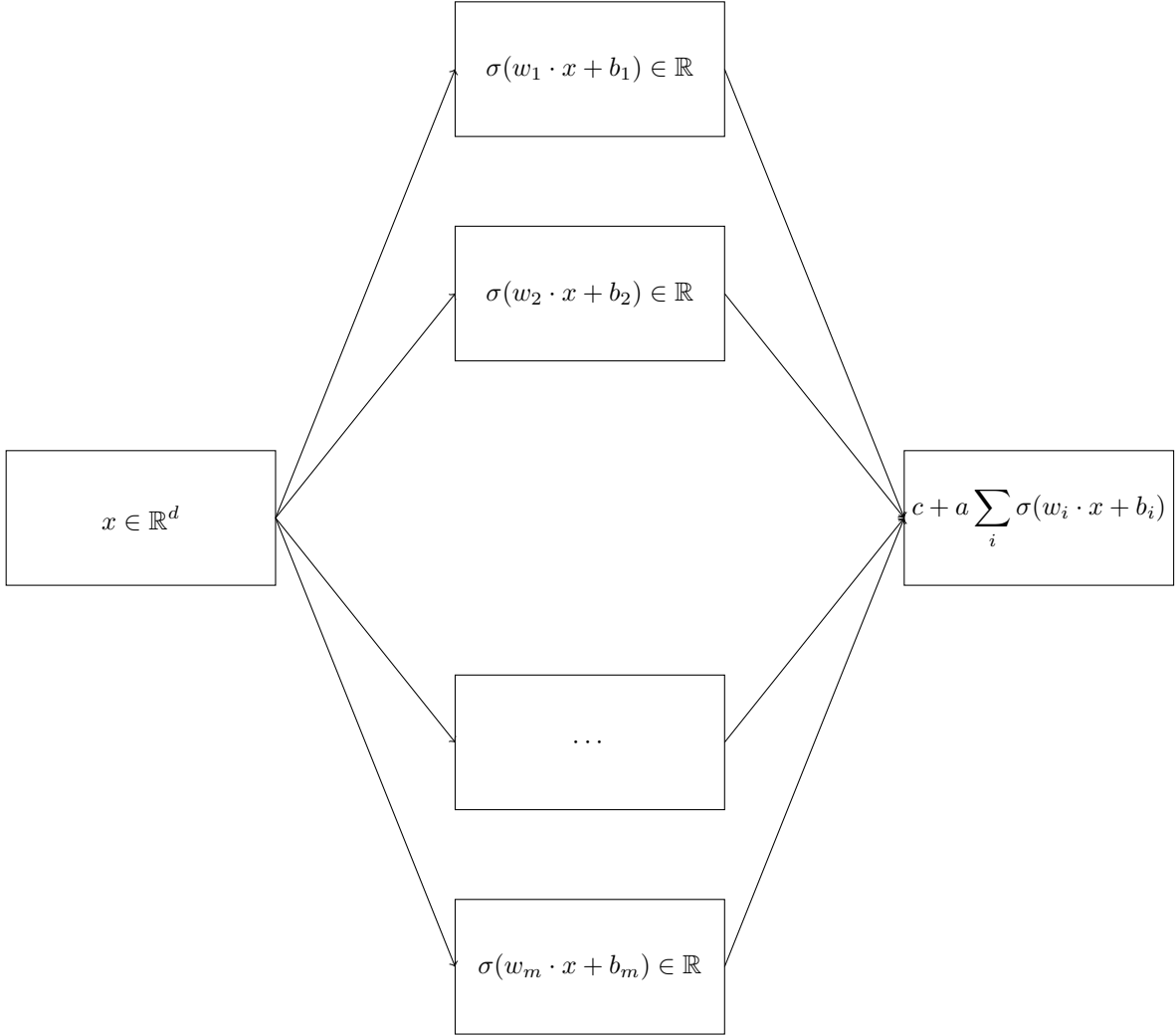
\begin{figure}[H]
\centering
\begin{tikzpicture}[scale=2]

\draw (-2*\D-\r,-\r/2) rectangle (-2*\D+\r,\r/2);
\draw (2*\D-\r,-\r/2) rectangle (2*\D+\r,\r/2);

\draw (-\r,\D-\r/2) rectangle (\r, \D+\r/2);
\draw (-\r,2*\D-\r/2) rectangle (\r, 2*\D+\r/2);
\draw (-\r,-\D-\r/2) rectangle (\r, -\D+\r/2);
\draw (-\r,-2*\D-\r/2) rectangle (\r, -2*\D+\r/2);

\draw[->] (-2*\D + \r ,0) -- (-\r,2*\D) ;
\draw[->] (-2*\D + \r ,0) -- (-\r,\D);
\draw[->] (-2*\D + \r ,0) -- (-\r,-\D);
\draw[->] (-2*\D + \r ,0) -- (-\r,-2*\D);

\draw[->] (\r,2*\D) -- (2*\D - \r ,0);
\draw[->] (\r,\D) -- (2*\D - \r ,0);
\draw[->] (\r,-\D) -- (2*\D - \r ,0);
\draw[->] (\r,-2*\D) -- (2*\D - \r ,0);

\draw (-2*\D,0) node {$x \in \R^d$};

\draw (0,2*\D) node {$\sigma(w_1 \cdot x + b_1) \in \R $};
\draw (0,\D) node {$\sigma(w_2 \cdot x + b_2)\in \R $};
\draw (0,-\D) node {$\cdots$};
\draw (0,-2*\D) node {$\sigma(w_m \cdot x + b_m) \in \R $};

\draw (2*\D,0) node {$c + a \sum_i \sigma(w_i \cdot x + b_i) $};

\end{tikzpicture}
\caption{A two-layer neural network of width $m$}
\label{figureDescriptionTwoLayersNN}
\end{figure}

\section{Preliminaries}\label{sec:prelim}

This section introduces the mathematical framework we consider in this paper to relate two-layer neural networks and high-dimensional Poisson equations. 

\subsection{Problem setting}

The following Poisson equation is considered on $\Omega := [0,1]^d$ ($d \in \mathbb{N}$) with Neumann boundary condition : find $u^* \in H^1(\Omega)$ with $\int_\Omega u^* = 0$ solution to :
\begin{equation}\label{PoissonEquation}
\left\{
\begin{aligned}
- \Delta u^\star &= f \text{ on } \Omega, \\
 \partial_n u^\star &= 0 \text{ on } \partial \Omega,
\end{aligned}
\right.
\end{equation}
where $f \in L^2(\Omega)$ with $\int_\Omega f=0$. Here~\eqref{PoissonEquation} has to be understood in the variational sense, in the sense that $u^*$ is equivalently the unique minimizer to :
\begin{equation}\label{eq:min}
u^\star = \argmin_{u \in H^1(\Omega)} \mathcal{E}(u), 
\end{equation}
where 
$$
\forall u \in H^1(\Omega), \quad \mathcal{E}(u):=  \int_\Omega \left(\frac{|\nabla u|^2}{2} - fu \right) dx + \frac{1}{2}\Big( \int_{\Omega} u dx \Big)^2.
$$
This can indeed be easily checked by classic Lax-Milgram arguments.
The functional $\E$ is strongly convex and differentiable with derivative given by Lemma~\ref{differentiabilityE}.

\begin{lemma}
\label{differentiabilityE}
The functional $\E : H^1(\Omega) \rightarrow \mathbb{R}$ is continuous, differentiable and for all $u \in H^1(\Omega)$, it holds that
$$
\forall v \in H^1(\Omega), \ d \E|_u(v) = \int_\Omega \left( \nabla u \cdot \nabla v- fv \right) dx + \int_{\Omega} u dx \int_{\Omega} v dx .
$$
\end{lemma}

It can be easily seen that points $u$ where the differential is identically zero  are solution to equation~\eqref{PoissonEquation}.

\begin{remark}\label{remarkEnergyBilinearForm}
The coercive symmetric bilinear form $\bar{a}$ involved in the definition of the energy writes :

$$
\bar{a}(u,v) := \int_\Omega \nabla u \cdot \nabla v  dx + \int_{\Omega} u dx \int_{\Omega} v dx.
$$
The energy $\E$ can then be equivalently rewritten thanks to the bilinear form $a$ : 

$$
\E(u) = \frac{1}{2} \bar{a}(u-u^\star, u-u^\star) - \frac{1}{2} \int_{\Omega} |\nabla u^\star|^2 dx.
$$
\end{remark}

The aim of the present work is to analyze a numerical method based on the use of infinite-width two-layer neural networks for the resolution of~\eqref{PoissonEquation} with a specific focus on the case when $d$ is large.

\subsection{Activation function}\label{sec:activation}

We introduce here the particular choice of activation function we consider in this work.

\medskip

Let $\sigma: \mathbb{R} \rightarrow \mathbb{R} $ be the classical Rectified Linear Unit (ReLU) function where :

\begin{equation}\label{definitionRelu}
\forall y \in \mathbb{R}, \ \sigma(y) := \max(y,0).
\end{equation}

Let $\rho: \mathbb{R} \to \mathbb{R}$ be defined by 
\begin{equation}
\left\{
\begin{array}{cl}
Z \exp\left(- \frac{\tan(\frac{\pi}{2} y)^2}{2}\right) & \text{if } |y| \leq 1 \\
0 & \text{otherwise}.
\end{array}
\right.
\end{equation}
where the constant $Z\in \mathbb{R}$ is defined such that the integral of $\rho$ is equal to one. For all $\tau >0$, we then define $\rho_\tau := \tau\rho(\tau \cdot)$ and $\sigma_\tau: \mathbb{R} \to \mathbb{R}$ the function defined by
\begin{equation}\label{definitionSoftplus}
\forall y \in \mathbb{R}, \ \sigma_\tau(y) := (\rho_\tau \star \sigma)(y). 
\end{equation}

We then have the following lemma. 
\begin{lemma}\label{lemmaApproxSigma}
For any $\tau >0$, it holds that
\begin{itemize}
\item[(i)] $\sigma_\tau\in {\mathcal C}^\infty(\mathbb{R})$ is uniformly bounded ans so is $\sigma^\prime_\tau$,
\item[(ii)] for all $y < -1/\tau$, $\sigma_\tau(y) = 0$,
\item[(iii)] for all $y > 1/\tau$, $\sigma_\tau(y) = y$,
\item[(iv)] there exists $C>0$ such that for all $\tau>0$, 
$$
\|\sigma - \sigma_\tau\|_{H^1(\mathbb{R})} \leq \frac{C}{\sqrt{\tau}}.
$$
\end{itemize}
\end{lemma}

\begin{proof}

The first item $(i)$ is classic and left to the reader. For $(ii)$, we have :

\begin{equation}\label{ecritureSigmaTau}
\sigma_\tau(y) = \int_{-1/\tau}^{1/\tau} \rho_\tau(y) \sigma(x-y) dy
\end{equation}
and if $x < -1/\tau$ then $x-y < 0$ for $-1/\tau<y<1/\tau$ and $\sigma(x-y) = 0$. This naturally gives $\sigma_\tau(y) = 0$.

For $(iii)$, using again~\eqref{ecritureSigmaTau} and if $x>1/\tau$, then $x-y > 0$ for $-1/\tau<y<1/\tau$ and $\sigma(x-y) = x-y$. As a consequence, 
$$
\sigma_\tau(y) =\int_{-1/\tau}^{1/\tau} \rho_\tau(y) (x-y) dy=x,
$$
where we have used the fact that $\int_{\R} \rho_\tau(y) dy = 1$ and $\int_{\R} y \rho_\tau(y) dy = 0$ by symmetry of $\rho$.

For $(iv)$, we have by $(ii)-(iii)$:
$$
\| \sigma - \sigma_\tau \|_{L^2(\R)}^2 = \int_{-1/\tau}^{1/\tau} (\sigma(x) - \sigma_\tau(x))^2 dx \leq \frac{8}{\tau^2},
$$
where we used the fact that $|\sigma(x)|, |\sigma_\tau(x)| \leq 1/\tau$ on $[-1/\tau,1/\tau]$. In a similar way, 
$$
\| \sigma^\prime - \sigma^\prime_\tau \|_{L^2(\R)}^2 = \int_{-1/\tau}^{1/\tau} (\sigma^\prime(x) - \sigma^\prime_\tau(x))^2 dx \leq \frac{4}{\tau}. 
$$
The two last inequalities gives $(iv)$.

\end{proof}

\medskip

In this work, we will rather use a hat version of the regularized ReLU activation function. More precisely, we define:
\begin{equation}\label{definitionRegularTrelu}
\forall y \in \mathbb{R}, \ \sigma_{H,\tau}(y) := \sigma_\tau(y+1) - \sigma_\tau(2y) + \sigma_\tau(y-1).
\end{equation}
We call hereafer this activation function the regularized HReLU (Hat ReLU) activation. When $\tau = +\infty$, the following notation is proposed :

\begin{equation}\label{definitionTrelu}
\forall y \in \mathbb{R}, \ \sigma_{H}(y) := \sigma(y+1) - \sigma(2y) + \sigma(y-1).
\end{equation}

The reasons why we use this activation is that it has a compact support and can be used to generate an arbitrary piecewise constant function on $[0,1]$. Note however that neither $\sigma_{H,\tau}$ nor $\sigma_H$ are homogeneous (in contrast to the activation functions considered in~\cite{BachChizat2018,BachChizat2021}). Notice also that a direct corollary of Lemma~\ref{lemmaApproxSigma} is that there exists a constant $C>0$ such that for all $\tau >0$, 
\begin{equation}\label{eq:sigmaerr}
\|\sigma_H - \sigma_{H,\tau}\|_{H^1(\mathbb{R})} \leq \frac{C}{\sqrt{\tau}}
\end{equation}
We will also use the fact that there exists a constant $C>0$ such that for all $\tau>0$, 
\begin{equation}\label{eq:bound}
\|\sigma_{H,\tau}\|_{L^\infty(\mathbb{R})}\leq C, \; \|\sigma'_{H,\tau}\|_{L^\infty(\mathbb{R})}\leq C,\; \|\sigma''_{H,\tau}\|_{L^\infty(\mathbb{R})}\leq C\tau \; \mbox{ and } \|\sigma'''_{H,\tau}\|_{L^\infty(\mathbb{R})}\leq C\tau^2. 
\end{equation}

\begin{figure}[H]
\centering
\includegraphics[width=0.5\textwidth]{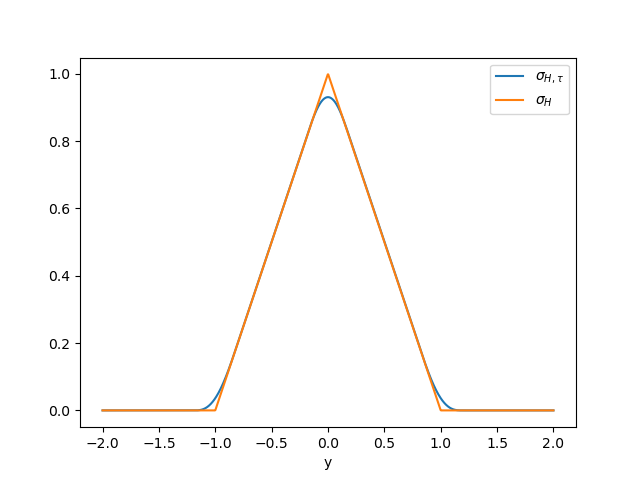}
\caption{The hat activation function and its regularization ($\tau = 4$)}
\end{figure}

\subsection{Spectral Barron space}\label{sectionSpectralBarron}

We introduce the orthonormal basis in $L^2(\Omega)$ composed of the eigenfunctions $\{\phi_k\}_{k\in \mathbb{N}^d}$ of the Laplacian operator with Neumann boundary conditions, where
\begin{equation}\label{baseL2}
\forall k = (k_1,\ldots, k_d)\in \N^d, \; \forall x:=(x_1,\cdots, x_d)\in \Omega, \quad  \phi_k(x_1, \ldots,x_d):= \prod_{i=1}^d \cos(\pi k_i x_i).
\end{equation}

Notice that $\{\phi_k\}_{k\in \N^d}$ is also an orthogonal basis of $H^1(\Omega)$. Using this basis, we have the Fourier representation formula for any function $u \in L^2(\Omega)$ :
$$
u = \sum_{k \in \N^d} \hat{u}(k) \phi_k,
$$
where for all $k\in \mathbb{N}^d$, $ \hat{u}(k):= \langle \phi_k, u \rangle_{L^2(\Omega)}$. This allows to define the (spectral) Barron space~\cite{Jianfeng2021} as follows :

\begin{definition}

For all $s>0$, the Barron space $\barron^s(\Omega)$ is defined as :

\begin{equation}
\label{DefinitionBarron}
\barron^s(\Omega) := \Big\{ u \in L^1(\Omega) : \sum_{k \in \N^d}(1 + \pi^s |k|^s_1)|\hat{u}(k)| < +\infty \Big\}
\end{equation}
and the space $\barron^2(\Omega)$ is denoted $\barron(\Omega)$. Moreover, the space $\barron^s(\Omega)$ is embedded with the norm :
\begin{equation}
\label{normBarron}
\| u\|_{\barron^s(\Omega)} := \sum_{k \in \N^d}(1 + \pi^s |k|^s_1)|\hat{u}(k)|.
\end{equation}

\end{definition}

By~\cite[Lemma 4.3]{Jianfeng2021}, it is possible to  relate the Barron space to traditional Sobolev spaces :

\begin{lemma}
The following continuous injections hold :
\begin{itemize}
\item $\barron(\Omega) \xhookrightarrow{} H^1(\Omega)$,
\item $\barron^0(\Omega) \xhookrightarrow{} L^\infty(\Omega)$.
\end{itemize}
\end{lemma}

The space $\barron(\Omega)$ has interesting approximation properties related to neural networks schemes. We introduce the following approximation space:
\begin{definition}
Let $\chi : \R \rightarrow \R$ be measurable, $m\in \N^*$ and $B>0$. The space $\F_{\chi,m}(B)$ is defined as:
\begin{equation}
\label{defFrelu}
\F_{\chi,m}(B) := \left\{ c + \sum_{i=1}^m a_i \chi(w_i \cdot x + b_i) : c, a_i, b_i\in \R, \, w_i \in \R^d, \;  |c| \leq 2B, |w_i| = 1, |b_i| \leq 1, \sum_{i = 1}^m |a_i| \leq 4B \right\}
\end{equation}

\end{definition}
Now, we are able to state the main approximation theorem.

\begin{theorem}
\label{approximationTheorem}

For any $u \in \barron(\Omega)$, $m \in \N^*$ :

\begin{itemize}
\item[(i)] there exists $u_m \in \F_{\sigma_H,m}(\| u \|_{\barron(\Omega)})$ such that :
$$
\| u - u_m \|_{H^1(\Omega)} \leq  \frac{C \| u \|_{\barron(\Omega)}}{\sqrt{m}},
$$

\item[(ii)] there exists $\tilde{u}_m \in \F_{\sigma_{H,m},m}(\| u \|_{\barron(\Omega)})$ such that :
\begin{equation}\label{eq:ineq}
\| u - \tilde{u}_m \|_{H^1(\Omega)} \leq  \frac{C \| u \|_{\barron(\Omega)}}{\sqrt{m}}.
\end{equation}

\end{itemize}
where for both items, $C$ is a universal constant which does not depend on $d$ neither on $u$.
\end{theorem}

\begin{proof}
Let $B:= \| u \|_{\barron(\Omega)}$.
We just give a sketch of the proof of (ii), (i) being derived from similar arguments as in~\cite[Theorem 2.1]{Jianfeng2021}

By (i), there exists $u_m \in \F_{\sigma_H,m}(B)$ such that
$$
\| u - u_m \|_{H^1(\Omega)} \leq \frac{C B }{\sqrt{m}}.
$$
The function $u_m$ can be written as :
$$
u_m(x) = c + \sum_{i=1}^m a_i \sigma_H(w_i \cdot x + b_i)
$$
for some $c, a_i, b_i\in \R$ , $w_i \in \R^d$ for $i=1,\ldots,m$ with $|c| \leq 2B, |w_i| = 1, |b_i| \leq 1, \sum_{i = 1}^m |a_i| \leq 4B$.

By Lemma~\ref{lemmaApproxSigma} $(iv)$, there exists $C>0$ such that for all $\tau>0$, $\| \sigma_H - \sigma_{H,\tau}\|_{H^1(\R)} \leq \frac{C}{\sqrt{\tau}}$, it is easy to see that 

$$
\| \tilde{u}_m - u_m \|_{H^1(\Omega)} \leq  \frac{CB}{\sqrt{m}}
$$
where :
$$
\tilde{u}_m(x) = c + \sum_{i=1}^m a_i \sigma_{H,m}(w_i \cdot x + b_i).
$$
Consequently,

$$
\| u - \tilde{u}_m \|_{H^1(\Omega)} \leq  \frac{C B}{\sqrt{m}}
$$
which yields the desired result.
\end{proof}

\begin{remark}
With other words, a Barron function can be approximated in $H^1(\Omega)$ by a two-layer neural network of width $m$
with precision $O\left(\frac{1}{\sqrt{m}}\right)$  when the activation function is the HReLU one. 
\end{remark}

In the sequel, we assume that any parameter vector $\theta = (c,a,w,b)$ takes values in the neural network parameter set
\begin{equation}\label{eq:defTheta}
\Theta := \R \times \R  \times S_{\R^d}(1) \times [-\sqrt{d} -2,\sqrt{d} + 2], 
\end{equation}
with $S_{\R^d}(1)$ the unit sphere of $\R^d$. In addition, for all $r>0$, we denote by 
\begin{equation}\label{eq:defKr}
K_r:= [-2r, 2r] \times [-4r, 4r] \times S_{\R^d}(1) \times  [-\sqrt{d} -2,\sqrt{d} + 2].
\end{equation}
The particular choice of the domain value of the parameter $b$, namely $[-\sqrt{d} -2,\sqrt{d} + 2]$ will be made clear in the following. Moreover, let $\Prob_2(\Theta)$ (respectively $\Prob_2(K_r)$) denote the set of probability measures on $\Theta$ (respectively on $K_r$) with finite second-order moments.

Let us make the following remark.
\begin{remark}\label{remarkUnifBoundFm}
Let $m\in \mathbb{N}^*$, $u_m \in \F_{\chi, m}(B)$ with $B>0$ and $\chi: \R \rightarrow \R$. Then, there exists $c, a_i, b_i\in \R$ , $w_i \in \R^d$ for $i=1,\ldots,m$ with $|c| \leq 2B, |w_i| = 1, |b_i| \leq 1, \sum_{i = 1}^m |a_i| \leq 4B$ such that for all $x\in \Omega$, 
\begin{align*}
u_m(x) &= c + \sum_{i=1}^m a_i \chi(w_i \cdot x + b_i)\\
&=  \sum_{i=1}^m \left( c + \sum_{j=1}^m |a_j| sign(a_i) \chi(w_i \cdot x + b_i) \right) \frac{|a_i|}{\sum_{j=1}^m |a_j|} \\
&=  \int_{\Theta}  [c + a \chi(w \cdot x +b)] d\mu_m(c,a,w,b),\\
\end{align*}
where the measure $\mu_m$ is a probability measure on $\Theta$ given by :
$$
\mu_m := \sum_{i=1}^m \frac{|a_i|}{\sum_{j=1}^m |a_j|} \delta_{(c,\sum_{j=1}^m |a_j| sign(a_i), w_i, b_i)}.
$$

Remark that $\mu_m$ has support in $K_B$. In addition, the sequence $(\mu_m)_{m\in \mathbb{N}^*}$ is  uniformly (with respect to $m$) bounded in $\Prob_2(\Theta)$.

\end{remark}

For a general domain $\Omega$ which is not of the form $\Omega = [0,1]^d$, the solution to equation~\eqref{PoissonEquation} does not necessarily belong to the Barron space even if the source term has finite Barron norm. Nevertheless for our case ($\Omega = [0,1]^d$), there is an explicit bound of the Barron norm of the solution compared with the source one. This gives hope for a neural network approximation of the solution.

\begin{theorem}\cite{Jianfeng2021}
\label{theoremBarronPoisson}
Let $u^*$ be the solution of the equation~\eqref{PoissonEquation} with $f \in \barron^0(\Omega)$,
then $u^* \in \barron(\Omega)$. Moreover, the following estimate holds :
$$
\| u^* \|_{\barron(\Omega)} \leq d \| f \|_{\barron^0(\Omega)}.
$$
\end{theorem}

\subsection{Infinite width two-layer neural networks}
In order to ease the notation for future computations, for all $\tau >0$, we introduce the function $\Phi_{\tau} : \Theta \times \Omega \to \mathbb{R}$ defined by

\begin{equation}
\label{definitionPhi}
\forall \theta:=(c,a,w,b)\in \Theta, \; \forall x\in \Omega, \quad \Phi_\tau(\theta; x) := c + a \sigma_{H,\tau}(w \cdot x + b)
\end{equation}
and $\Phi_{\infty}: \Theta \times \Omega \to \mathbb{R}$ defined by such that:

\begin{equation}
\label{definitionPhi1}
\forall \theta:=(c,a,w,b)\in \Theta, \; \forall x\in \Omega, \quad\Phi_\infty(\theta; x) := c + a \sigma_H(w \cdot x + b).
\end{equation}
The space $\Prob_2(\Theta)$ is embedded with the $2$-Wasserstein distance :
$$
\forall \mu, \nu \in \Prob_2(\Theta), \quad W^2_2(\mu,\nu) := \inf_{\gamma \in \Gamma(\mu,\nu)} \int_{\Theta^2} d(\theta,\tilde{\theta})^2 d \gamma(\theta,\tilde{\theta}),
$$
where $\Gamma(\mu,\nu)$ is the set of probability measures on $\Theta^2$ with marginals given respectively by $\mu$ and $\nu$ and where $d$ is the geodesic distance in $\Theta$. For the interested reader, the geodesic distance between $\theta, \tilde{\theta} \in \Theta$ can be computed as :

$$
d(\theta, \tilde{\theta}) = \sqrt{(c - \tilde{c})^2 + (a - \tilde{a})^2 + d_{S_{\R^d}(1)}(w,\tilde{w}) + (b - \tilde{b})^2}.
$$

For all $\tau, r>0$, we introduce the operator $P_{\tau}$ and the functional $\mathcal{E}_{\tau,r}$ defined as follows :
\begin{definition}
The operator $P_{\tau}: \Prob_2(\Theta) \rightarrow H^1(\Omega)$ is defined for all
$\mu \in \Prob_2(\Theta)$ as :

$$
P_{\tau}(\mu) := \int_{\Theta} \Phi_{\tau}(\theta; x) d\mu(\theta).
$$
Additionally, we define the functional $\mathcal{E}_{\tau,r}(\mu) : \Prob_2(\Theta) \rightarrow \mathbb{R}$ as :
$$
\mathcal{E}_{\tau,r}(\mu) := 
\left\{
\begin{aligned}
\E(P_{\tau}(\mu)) & \text{ if } \mu(K_{r}) = 1\\
+\infty & \text{ otherwise}.
\end{aligned}
\right.
$$. 
\end{definition} 

\begin{proposition}\label{propositionLowerSemicontinuity}
For all $0<\tau,r<\infty$, the functional $\mathcal{E}_{\tau,r}$ is weakly lower semicontinuous.
\end{proposition}

\begin{proof}
Let $(\mu_n)_{n\in \mathbb{N}^*}$ be a sequence of elements of $\Prob_2(\Theta)$ which narrowly converges towards some $\mu \in \Prob_2(\Theta)$. Without loss of generality, we can assume that $\mu_n$ is supported in $K_{r}$ for all $n\in \mathbb{N}^*$. Then, it holds that : 

\begin{itemize}

\item the limit $\mu$ has support in $K_{r}$ (by Portmanteau theorem);
\item moreover, let $u_n:\Omega \to \mathbb{R}$ be defined such that for all $x\in \Omega$, 
$$
u_n(x) := \int_{\Theta} \Phi_{\tau}(\theta; x) d\mu_n(\theta) = \int_{K_r} \Phi_{\tau}(\theta; x) \,d\mu_n(\theta).
$$
Since for all $x\in \Omega$, the function  $K_{r} \ni \theta \mapsto \Phi_{\tau}(\theta; x)$ is continuous and bounded, it then holds that, for all $x\in \Omega$,
$$
u_n(x) \mathop{\longrightarrow}_{n \to \infty} u(x) :=  \int_{K_r} \Phi_{\tau}(\theta; x) d\mu(\theta) = \int_{\Theta} \Phi_{\tau}(\theta; x) d\mu(\theta), 
$$
where the last equality comes from the fact that $\mu$ is supported in $K_{r}$.

\item It actually holds that the sequence $(u_n)_{n\in \mathbb{N}^*}$ is uniformly bounded in ${\mathcal C}(\Omega)$. Indeed, there exists $C_\tau>0$ such that for all $x\in \Omega$ and $n\in \mathbb{N}^*$, we have 
\begin{align*}
u_n(x)^2  &=  \left( \int_{K_{r}} \Phi_{\tau}(\theta; x) d\mu_n(\theta) \right)^2 \\
& \leq   \int_{K_{r}} \Phi_{\tau}^2(\theta; x) d\mu_n(\theta) \\
& \leq  C r^2,\\
\end{align*}
where the last inequality comes from~\eqref{eq:bound}. 
\end{itemize}

As a consequence of the Lebesgue dominated convergence theorem, the sequence $(u_n)_{n\in \mathbb{N}^*}$ strongly converges towards $u$ in $L^2(\Omega)$. Reproducing the same argument as above for the sequence $(\nabla u_n)_{n\in \mathbb{N}^*}$, one easily proves that this strong convergence holds in fact in $H^1(\Omega)$. The fact that the functional $\E : H^1(\Omega) \rightarrow \R$ is continuous allows us to conclude.
\end{proof}

\begin{remark}\label{remarkLowerSemicontinuity}
In $\Prob_2(K_r)$, the weak convergence is metricized by the Wasserstein distance. Hence, $\E_{\tau}$ is lower semicontinuous 
as a functional from $(\Prob_2(\Theta), W_2)$ to $(\R, |\cdot|)$.
\end{remark}

Finally, the lower semicontinuity of $\E_{\tau,r}$ and compactness of $\Prob_2(K_{r})$ (as $K_r$ is compact) allows to prove the existence of at least one solution to the following minimization problem :

\begin{problem}\label{poissonProba}
For $0<\tau<\infty$ and $0<r <+\infty$, let $\mu_{\tau,r}^\star \in \Prob_2(\Theta)$ be solution to
\begin{equation}\label{eq:tauprob}
\mu_{\tau,r}^\star \in \argmin_{\mu \in \Prob_2(\Theta)} \E_{\tau,r}(\mu).
\end{equation}
\end{problem}

For large values of $\tau$ and $r = d\|f\|_{\barron^0(\Omega)}$, solutions of~\eqref{eq:tauprob} enable to obtain accurate approximations of the solution of~\eqref{PoissonEquation}. This result is stated in
Theorem~\ref{theoremEquivalencePoissonP2}.

\begin{theorem}
\label{theoremEquivalencePoissonP2}
There exists $C>0$ such that for all $m\in \mathbb{N}^*$ and any solution $\mu^\star_{m, d\|f\|_{\barron^0(\Omega)}}$ to~\eqref{eq:tauprob} with $\tau = m$ and $r = d\|f\|_{\barron^0(\Omega)}$, it holds that:  
$$
\Big\| u^\star - \int_{\Theta} \Phi_{m}(\theta; \cdot) d\mu_{m,d\|f\|_{\barron^0(\Omega)} }^\star(\theta) \Big\|_{H^1(\Omega)} \leq Cd\frac{\|f\|_{\barron^0(\Omega)} }{\sqrt{m}}
$$
where $u^\star$ is the solution of the equation~\eqref{PoissonEquation}.
\end{theorem}

\begin{proof}
For all $m\in \mathbb{N}^*$, let $\tilde{u}_m\in \F_{\sigma_{H,m},m}(\| u^* \|_{\barron})$ satisfying (\ref{eq:ineq}) for $u=u^*$ (using Theorem~\ref{approximationTheorem}). Since $\|u^*\|_{\barron(\Omega)}\leq d \|f\|_{\barron^0(\Omega)}$ thanks to Theorem~\ref{theoremBarronPoisson} and by Remark~\ref{remarkUnifBoundFm}, $\tilde{u}_m$ can be rewritten using a probability measure $\mu_m$ with support in $K_{d\|f\|_{\barron^0(\Omega)}}$ as : 
$$
\forall x\in \Omega, \quad \tilde{u}_m(x) = \int_{\Theta} \Phi_{m}(\theta; x) \,d\mu_m(\theta).
$$
Let $\mu^\star_{m, d\|f\|_{\barron^0(\Omega)}}$ be a minimizer of (\ref{eq:tauprob}) with $\tau = m$ and $r = d\|f\|_{\barron^0(\Omega)}$. Then, it holds that:
$$
\E_{m, d\|f\|_{\barron^0(\Omega)}}\left(\mu^\star_{m, d\|f\|_{\barron^0(\Omega)}}\right) \leq \E_{m, d\|f\|_{\barron^0(\Omega)}}(\mu_m),
$$
which by Remark~\ref{remarkEnergyBilinearForm}, is equivalent to :
$$
\bar{a}(u^\star_m-u^\star,u^\star_m-u^\star) \leq \bar{a}(\tilde{u}_m-u^\star,\tilde{u}_m-u^\star).
$$
where for all $x\in \Omega$, 
$$
u^\star_m(x):= \int_\Theta \Phi_m(\theta;x)\,d\mu^\star_{m, d\|f\|_{\barron^0(\Omega)}}(\theta).
$$
Denoting by $\alpha$ and $L$ respectively the coercivity and continuity constants of $\bar{a}$, we obtain that
$$
\| u^\star_m-u^\star \|_{H^1(\Omega)} \leq \frac{L}{\alpha} \| \tilde{u}_m-u^\star \|_{H^1(\Omega)} \leq Cd\frac{\|f\|_{\barron^0(\Omega)} }{\sqrt{m}}.
$$
\end{proof}

\subsection{Main results}\label{sectionMainResults}

In this section, we find a solution to Problem~\ref{poissonProba} using gradient curve techniques. More particularly, we will define and prove the existence of a gradient descent curve such that if the convergence is asserted, then the convergence necessarily holds towards a global minimizer. In all the sequel, we fix an a prioori chosen value of $\tau>0$.

\subsubsection{Well-posedness}
First, we introduce the concept of gradient curve which formally writes for $r >0$:
\begin{equation}\label{eq:formal}
\forall t\geq 0, \ \frac{d}{dt} \mu^{r}(t) = - \nabla \E_{\tau,r}(\mu^{r}(t)).
\end{equation}
Equation~\eqref{eq:formal} has no mathematical sense since the space $\Prob_2(\Theta)$ is not a Hilbert space and consequently, the gradient of $\E_{\tau,r}$ is not available in a classical sense. Nevertheless $\Prob_2(\Theta)$ being an Alexandrov space, it has a differential structure which allows to define gradients properly. The careful reader wishing to understand this structure can find a complete recap of all useful definitions and properties of Alexandrov spaces in Appendix~\ref{appendix}. 

Before exposing our main results of well-posedness, we recall the basic definition of local slope~\cite{AmbrosioSavareBook}. In the sequel, we denote by  $\Prob_2(K_r)$ the set of probability measures on $\Theta$ with support included in $K_r$.

\begin{definition}\label{definitionLocalSlope}
At every $\mu \in \Prob_2(K_r)$, the local slope writes :

$$
|\nabla^- \E_{\tau,r}|(\mu) := \limsup_{\nu \rightarrow \mu} \frac{(\E_{\tau,r}(\mu) - \E_{\tau,r}(\nu))_+}{W_2(\mu,\nu)}
$$
which may be infinite.
\end{definition}

In Section~\ref{sectionWellPosedness}, we prove two theorems; the first one states the existence and the uniqueness of the gradient curve with respect to $\E_{\tau,r}$ when $r<\infty$.

\begin{theorem}\label{theoremWellPosednessCondSupport}
For all $\mu_0 \in \Prob_2(K_r)$, there exists a unique locally Lipschitz gradient curve $\mu^{r}: \mathbb{R}_+ \to \Prob_2(K_r)$ which is also a curve of maximal slope with respect to the upper gradient $|\nabla^- \E_{\tau,r}|$. Moreover, for almost all $t\geq 0$, there exists a vector field $v^{r}_t\in L^2(\Theta;\,d\mu^r(t))^{d+3}$ such that
\begin{equation}
\int_{\Theta} \|v^{r}_t\|^2 \,d\mu^r(t) = \| v^{r}_t \|^2_{L^2(\Theta;\,d\mu^r(t))} < +\infty
\end{equation}
and :
\begin{equation}
\label{wellPosednessEr}
\left\{
\begin{array}{rl}
\partial_t \mu^{r}(t) + \text{\rm div}(v^{r}_{t} \mu^{r}(t)) =& 0 \\
\mu^{r}(0) =& \mu_0\\
\mu^{r}(t) \in& \Prob_2(K_r).
\end{array}
\right.
\end{equation}
\end{theorem}

In the second theorem, we focus on the case when $r= + \infty$ for which we formally take the limit of gradient curves $(\mu^{r})_{r>0}$ as $r$ goes to infinity. Introducing the following quantities, the definition of which will be made precise below :
$$
\left\{
\begin{array}{rl}
\phi_\mu(\theta) :=& d\E|_{P_{\tau}(\mu)}(\Phi_{\tau}(\theta;\cdot)),\\
v_\mu(\theta) :=& \nabla_\theta \phi_\mu(\theta), \\
\end{array}
\right.
$$
and $\bold{P}$ the projection on the tangent bundle of $\Theta$ the precise definition of which is given in Definition~\ref{def:P}, 
the following theorem is proved.
\begin{theorem}\label{theoremGlobalWellPosedness}
For all $\mu_0$ compactly supported, there exists a curve $\mu: \mathbb{R}_+ \to \Prob_2(\Theta)$ such that :

\begin{equation}
\label{wellPosednessE}
\left\{
\begin{aligned}
\partial_t \mu(t) + {\rm div}((-\bold{P} v_{\mu(t)}) \mu(t)) &= 0 \\
\mu(0) &= \mu_0
\end{aligned}
\right.
\end{equation}
and for almost all $t\geq 0$ :
$$
\int_{\Theta} |\bold{P} v_{\mu(t)}|^2 \; d \mu(t) = \| \bold{P} v_{\mu(t)} \|^2_{L^2(\Theta;d\mu(t))} < +\infty.
$$
Moreover, the solution satisfies :

$$
\forall t \geq 0, \mu(t) = {\chi(t)}\# \mu_0
$$
with $\chi: \mathbb{R}_+ \times \Theta \to \Theta$ solution to
$$
\left\{
\begin{aligned}
\partial_t \chi(t;\theta) &= -\bold{P} v_{\mu(t)}(\theta)\\
\chi(0;\theta) &= \theta.
\end{aligned}
\right.
$$
\end{theorem}

In Remark~\ref{remarkNonexistenceUniqueness}, we argue why proving the existence and uniqueness of a gradient curve for $\E_{\tau,\infty}$ is not reachable. This is why $\mu: \mathbb{R}_+ \to \Prob_2(\Theta)$ is described as a limiting gradient curve and not a gradient curve itself in Theorem~\ref{theoremGlobalWellPosedness}.  

\subsubsection{Link with neural network}
Our motivation for considering the analysis presented in the previous section is that we can link the learning phase of a neural network with the optimization procedure given by gradient curves defined above. Indeed, let $m>0$ be an integer. A two-layer neural network $u$ with $\sigma_{H,\tau}$ as activation function can always be written as :

\begin{equation}\label{network}
u = \frac{1}{m} \sum_{i=1}^m \Phi_\tau(\theta_i,\cdot)
\end{equation}
with $\theta_i \in \Theta$. Then, we differentiate the functional $\Fu :(\theta_1, \cdots, \theta_m) \rightarrow \E \left(\frac{1}{m} \sum_{i=1}^m \Phi_\tau(\theta_i,\cdot)\right)$ :

$$
\begin{aligned}
d \Fu|_{\theta_1,\cdots, \theta_m}(d\theta_1, \cdots, d\theta_m) &= d \E|_u \left( \frac{1}{m} \sum_{i=1}^m \nabla_\theta \Phi_\tau(\theta_i,\cdot) \cdot d \theta_i \right).
\end{aligned}
$$
Thus, the gradient of $\Fu$ is given by :

$$
\begin{aligned}
\nabla_{\theta_i} \Fu(\theta_1,\cdots, \theta_m)
& = \frac{1}{m} \nabla_{\theta} \phi_{\mu}(\theta_i)
\end{aligned}
$$
where :

\begin{equation}\label{equivalentMu}
\mu := \frac{1}{m} \sum_{i=1}^m \delta_{\theta_i} \in \Prob_2(\Theta).
\end{equation}
As a consequence, a gradient descent of $\F$ in the sense that, for all $1\leq i \leq m$, 
$$
\left\{
\begin{aligned}
\frac{d}{dt}\theta_i(t) & =  - m \bold{P} \nabla_{\theta_i} \F(\theta_1(t), \cdots, \theta_m(t)) \\
\theta_i(0) & =  \theta_{i,0}, \\
\end{aligned}
\right.
$$
which is equivalent to the gradient curve of $\E_{ \tau, +\infty}$ with initial condition given by
\begin{equation}\label{eq:Dirac}
\mu_{0,m}:=\frac{1}{m} \sum_{i=1}^m \delta_{\theta_{i,0}}.
\end{equation}

\begin{theorem}\label{theoremLinkNN}
Let $\mu_0 \in \Prob_2(\Omega)$ compactly supported, $(\mu_{0,m})_{m\in \mathbb{N}^*}$ be such that for all $m\in \mathbb{N}^*$, $\mu_{0,m}$ is of the form~\eqref{eq:Dirac} for some $(\theta_{i,0})_{1\leq i \leq m} \subset {\rm Supp}(\mu_0)$ and $\lim_{m \rightarrow + \infty} W_2(\mu_{0,m}, \mu_0) = 0$. 

Let $\mu: \mathbb{R}_+ \to \Prob_2(\Theta)$ and $\mu_m: \mathbb{R}_+ \to \Prob_2(\Theta)$ be the  gradient curves constructed in Theorem~\ref{theoremGlobalWellPosedness} associated respectively to the initial conditions $\mu(0)=\mu_0$ and $\mu_m(0) = \mu_{0,m}$. Then for all $T>0$, there exists a constant $C_{T}>0$ such that 
$$
\mathop{\sup}_{0\leq t \leq T} W_2(\mu(t), \mu_m(t)) \leq C_T W_2(\mu_0, \mu_{0,m}).
$$
\end{theorem}
This theorem is proved in Section~\ref{sectionLinkNN}.

\subsubsection{Convergence}

Our convergence result towards a global optimum is based on the following hypothesis on the initial measure $\mu_0$ : 

\begin{hypothesis}\label{hypothesisSupport}
The support of the measure $\mu_0$ verifies :
$$
\{0\} \times \{0\} \times S_{\R^d}(1) \times [-\sqrt{d} -2,\sqrt{d} + 2] \subset {\rm Supp}(\mu_0)
$$
\end{hypothesis}

Under such hypothesis, one gets a result of convergence in the spirit of a previous work from Bach and Chizat~\cite{BachChizat2018} :
\begin{theorem}\label{theoremConvergence}
If $\mu_0$ satisfies Hypothesis~\ref{hypothesisSupport} and $\mu(t)$ converges towards $\mu^\star \in \Prob_2(\Theta)$ as $t$ goes to infinity, then $\mu^\star$ is optimal for Problem~\ref{poissonProba}.
\end{theorem}

This theorem is proved in Section~\ref{sectionConvergence}.

\section{Gradient curve}\label{sectionGradientFlow}

This section is dedicated to the proof of the two theorems stated in Section~\ref{sectionMainResults}.

\subsection{Well-posedness}\label{sectionWellPosedness}

\subsubsection{Proof of Theorem~\ref{theoremWellPosednessCondSupport}}

Let us fix some value of $r>0$ in this section. In the following, $C>0$ will denote an arbitrary constant which does not depend on $\tau$ and $r$. Let $\mathfrak{P}$ be the set of geodesics of $\Theta$ \textit{ie} the set of absolutely continuous curves $\pi : [0,1] \rightarrow \Theta$ such that for all $t_1,t_2\in[0,1]$, $d(\pi(t_1), \pi(t_2)) = d(\pi(0), \pi(1))|t_1 - t_2|$. Besides, it holds that for all $0\leq t \leq 1$, we have $|\dot{\pi}(t)|= d(\pi(0), \pi(1))$. 

For all $s\in [0,1]$, we define the application map $e_s : \mathfrak{P} \rightarrow \Theta$ such that $e_s(\pi) := \pi(s)$. Owing this, McCann interpolation gives the fundamental characterization of constant speed geodesics in $\Prob_2(\Theta)$ :

\begin{proposition}\cite[Proposition 2.10]{LottVillani2009}\label{propositionMccannInterpolant}
For all $\mu, \nu \in \Prob_2(\Theta)$ and any geodesic $\kappa: [0,1] \to \Prob_2(\Theta)$ between them (i.e. such that $\kappa(0) = \mu$ and $\kappa(1) = \nu$) in the $W_2$ sense, there exists $\Pi \in \Prob_2(\mathfrak{P})$ such that :

$$
\forall t \in [0,1], \ \kappa(t) = {e_t}\# \Pi.
$$ 
\end{proposition}

\begin{remark}\label{remarkSupportGeodesic}
As ${e_0}\# \Pi = \mu$ and ${e_1}\# \Pi = \nu$, the support of $\Pi$ is included in the set of geodesics $\pi: [0,1] \to \Theta$ such that $\pi(0)$ belongs to the support of $\mu$ and $\pi(1)$ belongs to the support of $\nu$. In addition, it holds that $\gamma:=(e_0,e_1)\#\Pi$ is then an optimal transport plan between $\mu$ and $\nu$ for the quadratic cost, i.e. $W_2(\mu^,\nu)^2 = \int_{\Theta \times \Theta} |\theta - \widetilde{\theta}|^2\,d\gamma(\theta, \widetilde{\theta})$. 
\end{remark}

The next result states smoothness properties of geodesics on $\Theta$ which are direct consequences of the smoothness properties of geodesics on the unit sphere of $\R^d$. It is a classical result and its proof is left to the reader. 
\begin{lemma}\label{lemmaSmoothnessGeodesic}
There exists $C>0$ such that for all $(\theta,\tilde{\theta})$ in $\Theta^2$, all geodesic $\pi: [0,1] \to \Theta$ such that $\pi(0) = \theta$ and $\pi(1) = \widetilde{\theta}$ and all $0 \leq s \leq t \leq 1$, 

$$
|\pi(t) - \pi(s)| \leq d(\pi (t), \pi (s)) = (t-s) d(\theta,\tilde{\theta}) \leq C (t-s) |\tilde{\theta} - \theta|
$$
and
$$
\left|\frac{d}{dt} \pi(t) \right| \leq d(\theta,\tilde{\theta}) \leq C |\tilde{\theta} - \theta|.
$$
\end{lemma}
In order to prove the well-posedness, it is necessary to get information about the smoothness of $\E_{\tau,r}$. 

\begin{proposition}
\label{lemmaDiffEp}

The functional $\mathcal{E}_{\tau,r}$ is proper, coercive, differentiable on $\Prob_2(K_r)$. Moreover, there exists a constant $C_{r,\tau}>0$ such that for all $\mu, \nu \in \Prob_2(K_r)$, $ \gamma \in \Gamma(\mu,\nu)$ with support included in $K_r \times K_r$:

\begin{equation}
\label{definitionV3}
\left| \mathcal{E}_{\tau,r}(\nu) - \mathcal{E}_{\tau,r}(\mu) + \int_{\Theta^2} v_\mu(\theta) \cdot(\tilde{\theta} - \theta) d\gamma(\theta, \tilde{\theta})\right| \leq  C_{r,\tau} c_2(\gamma)
\end{equation}
with
$$
c_2(\gamma) := \int_{\Theta ^2} (\theta-\tilde{\theta})^2\,d\gamma(\theta,\tilde{\theta}), 
$$
and
\begin{equation}
\label{definitionV1}
v_\mu(\theta) := \nabla_\theta \phi_\mu(\theta) 
\end{equation}
where for all $\theta \in K_r$,
\begin{equation}
\label{definitionV2}
\begin{aligned}
\phi_\mu(\theta) & :=  \langle \nabla_x P_{\tau}(\mu), \nabla_x \Phi_{\tau}(\theta;\cdot) \rangle_{L^2(\Omega)} - \langle f,  \Phi_{\tau}(\theta;\cdot)   \rangle_{L^2(\Omega)} + \int_{\Omega}  P_{\tau}(\mu)(x) dx \times \int_{\Omega} \Phi_{\tau}(\theta;x) dx\\
& = d\E|_{P_{\tau}(\mu)}(\Phi_{\tau}(\theta;\cdot)).
\end{aligned}
\end{equation}
\end{proposition}

The properness and coercivity are easy to prove and left to the reader. Before proving the differentiability property of $\E_{\tau, r}$, we will need the following auxiliary lemma. 

\begin{lemma}\label{lem:aux}
There exists a constant $C>0$ such that for all $\tau>0$ and all $\theta \in \Theta$, we have
\begin{align*}
\left\|\Phi_\tau(\theta; \cdot)\right\|_{L^\infty(\Omega)} & \leq C |\theta|,\\
\left\| \nabla_x \Phi_\tau(\theta; \cdot)\right\|_{L^\infty(\Omega)} & \leq C |\theta|,\\
\left\| \nabla_\theta \Phi_\tau(\theta; \cdot)\right\|_{L^\infty(\Omega)} & \leq C |\theta|,\\
\left\| H_\theta \Phi_\tau(\theta; \cdot)\right\|_{L^\infty(\Omega)} & \leq C |\theta| \tau,\\
\left\|\nabla_x \nabla_\theta \Phi_\tau(\theta; \cdot)\right\|_{L^\infty(\Omega)} & \leq C|\theta|\tau,\\
\left\| \nabla_x H_\theta  \Phi_\tau(\theta; \cdot)\right\|_{L^\infty(\Omega)} & \leq C |\theta| \tau^2,\\
\end{align*}
where for all $\theta \in \Theta$ and $x\in \Omega$, $H_\theta \Phi_\tau(\theta; x)$ denotes the Hessian of $\Phi_\tau$ with respect to the variable $\theta$ at the point $(\theta,x)\in \Theta \times \Omega$.  
\end{lemma}

\begin{proof}
Let $\theta = (c,a,w,b)\in \Theta$. It then holds that, for all $x\in \Omega$,
\begin{equation}
\label{gradtheta}
\left\{
\begin{aligned}
\frac{\partial \Phi_\tau(\theta; x)}{\partial c} &= 1\\ 
\frac{\partial \Phi_\tau(\theta; x)}{\partial a} &= \sigma_{H,\tau}(w\cdot x + b)\\
\frac{\partial \Phi_\tau(\theta; x)}{\partial w} &= ax\sigma^\prime_{H,\tau}(w\cdot x + b)\\
\frac{\partial \Phi_\tau(\theta; x)}{\partial b} &= a\sigma^\prime_{H,\tau}(w\cdot x + b).\\ 
\end{aligned}
\right.
\end{equation}
This expression yields the first desired inequality. In addition, the nonzero terms of the Hessian matrix read as:
\begin{equation}
\label{grad^2theta}
\left\{
\begin{aligned} 
\frac{\partial^2 \Phi_\tau(\theta; x)}{\partial a \partial w} &= \sigma_{H,\tau}^\prime(w\cdot x + b) x\\
\frac{\partial^2 \Phi_\tau(\theta; x)}{\partial a \partial b} &= \sigma_{H,\tau}^\prime(w\cdot x + b)\\
\frac{\partial^2 \Phi_\tau(\theta; x)}{\partial^2 w} &= a\sigma^{\prime \prime}_{H,\tau}(w\cdot x + b)x x^T\\
\frac{\partial^2 \Phi_\tau(\theta; x)}{\partial w \partial b} &= a\sigma^{\prime \prime}_{H,\tau}(w\cdot x + b)x\\
\frac{\partial^2 \Phi_\tau(\theta; x)}{\partial^2 b} &= a\sigma^{\prime \prime}_{H,\tau}(w\cdot x + b).
\end{aligned}
\right.
\end{equation}

From these expressions, together with (\ref{eq:bound}), we easily get that, for all $\theta \in K_r$, 
$$
\left\| H_\theta \Phi_\tau(\theta; \cdot)\right\|_{L^\infty(\Omega)} \leq C r \tau,
$$
for some constant $C>0$ independent of $\theta$, $r$ and $\tau$.
Moreover, for all $x\in \Omega$,
\begin{equation}
\label{gradx}
\nabla_x \Phi_\tau(\theta; x) = a w \sigma^\prime_{H,\tau} (w \cdot x + b),
\end{equation}
which implies that
\begin{equation}
\label{gradthetax}
\left\{
\begin{aligned}
 \frac{\partial \nabla_x\Phi_\tau(\theta; x)}{\partial c} &= 0\\ 
  \frac{\partial \nabla_x\Phi_\tau(\theta; x)}{\partial a} &= w \sigma^\prime_{H,\tau}(w\cdot x + b)\\
 \frac{\partial \nabla_x \Phi_\tau(\theta; x)}{\partial w} &= a\sigma^\prime_{H,\tau}(w\cdot x + b) I_d + ax w^T\sigma^{\prime \prime}_{H,\tau}(w\cdot x + b) \\
\frac{\partial \nabla_x  \Phi_\tau(\theta; x)}{\partial b} &= a w \sigma^{\prime \prime}_{H,\tau}(w\cdot x + b).
\end{aligned}
\right.
\end{equation}
This implies then that
$$
\left\|\nabla_\theta \nabla_x\Phi_\tau(\theta; \cdot)\right\|_{L^\infty(\Omega)} \leq Cr\tau.
$$

Moreover, it then holds, using again (\ref{eq:bound}), that for all $\theta \in K_r$,
$$
\left\| H_\theta \nabla_x\Phi_\tau(\theta; \cdot)\right\|_{L^\infty(\Omega)} \leq C r \tau^2,
$$ 
for some constant $C>0$ independent of $\theta$, $r$ and $\tau$.

\end{proof}
The following corollary is also a prerequisite for the proof of Proposition~\ref{lemmaDiffEp}.

\begin{corollary}\label{cor:boundPtau}
There exists a constant $C_\tau>0$ and a constant $C_{r,\tau}>0$ such that for all $\mu, \nu \in \Prob_2(K_r)$ :
\begin{equation}\label{BoundPbarronmu}
\|  P_\tau(\mu) \|^2_{H^1(\Omega)} \leq C_\tau \int_{\Theta} |\theta|^2 d\mu(\theta),
\end{equation}
and
$$
\|  P_\tau(\mu) - P_\tau(\nu)\|^2_{H^1(\Omega)} \leq C_{r,\tau} W^2_2(\mu,\nu).
$$
\end{corollary}

\begin{proof}
From Lemma~\ref{lem:aux} we immediately obtain that, for all $\tau,r>0$, there exists a constant $C_{\tau,r}>0$ such that for all $\theta_1, \theta_2 \in K_r$, 
$$
\left\{
\begin{array}{rl}
\| \nabla_\theta \Phi_{\tau}(\theta_1 ;\cdot)\|_{ H^1(\Omega)} \leq & C_{\tau,r} |\theta_1|, \\
\|\nabla_\theta \Phi_{\tau} (\theta_1 ;\cdot) - \nabla_\theta \Phi_{\tau}(\theta_2 ;\cdot)\|_{ H^1(\Omega)} \leq &C_{\tau,r} |\theta_1 - \theta_2|.
\end{array}
\right.
$$
The corollary immediately follows from that fact. 

\end{proof}

Now we are able to prove Proposition~\ref{lemmaDiffEp}.

\begin{proof}

First, we focus on the proof of~\eqref{definitionV3}-\eqref{definitionV2}.
As $\Phi_{\tau}$ and $\E$ are smooth, it holds that for all $x\in \Omega$, $\theta, \widetilde{\theta}\in \Theta$, $u, \tilde{u}\in H^1(\Omega)$,

$$
\left\{
\begin{aligned}
\Phi_\tau(\tilde{\theta};x) &= \Phi_{\tau}(\theta;x) + \nabla_\theta \Phi_{\tau}(\theta;x) \cdot (\tilde{\theta} - \theta) + M_\tau(\theta, \tilde{\theta};x)\\
\E(\tilde{u}) &= \E(u) + d\E|_{u} (\tilde{u} - u) + N(\tilde{u} - u),
\end{aligned}
\right.
$$
where $N(u) := \frac{1}{2}\bar{a}(u,u)$ for all $u\in H^{1}(\Omega)$ and 
$
M_\tau(\theta, \tilde{\theta};x) :=  \int_{0}^1 (\tilde{\theta} - \theta)^T H_\theta \Phi_\tau(\theta + t (\tilde{\theta} - \theta);x) (\tilde{\theta} - \theta) (1-t) dt.
$ 
Using Lemma~\ref{lem:aux}, there exists a constant $C>0$ independent on $r$ and $\tau$ such that:
\begin{itemize}

\item $\forall x \in \Omega, \; \forall\theta, \tilde{\theta} \in K_r,  \ |M_\tau(\theta, \tilde{\theta};x)| \leq C r \tau |\theta - \tilde{\theta}|^2$,

\item $\forall x \in \Omega, \; \forall \theta, \tilde{\theta} \in K_r, \ |\nabla_x M_\tau(\theta, \tilde{\theta};x)| \leq C r \tau^2 |\theta - \tilde{\theta}|^2$.

\end{itemize}
Moreover, there exists a constant $C>0$ such that for all $u\in H^1(\Omega)$,
\begin{equation}
\label{boundN}
0 \leq N(u)  \leq C \| u \|^2_{H^1(\Omega)}.
\end{equation}

Thus, for $\mu,\nu \in \Prob_2(K_r)$ and $\gamma \in \Gamma(\mu,\nu)$ supported in $K_r^2$, it holds that:
$$
\begin{aligned}
\mathcal{E}_{\tau,r}(\nu) & =  \E\Big(\int_{K_r} \Phi_\tau(\tilde{\theta};\cdot) d\nu(\tilde{\theta}) \Big)\\
& = \E\Big(\int_{K_r^2}  \Phi_\tau(\tilde{\theta};\cdot) d\gamma(\theta, \tilde{\theta}) \Big) \\
& = \E\Big(\int_{K_r^2} \left[  \Phi_{\tau}(\theta;\cdot) + \nabla_\theta \Phi_{\tau}(\theta;\cdot)\cdot  (\tilde{\theta} - \theta) + M_\tau( \theta, \tilde{\theta}; \cdot)\right] d\gamma(\theta, \tilde{\theta}) \Big)\\
& = \mathcal{E}_{\tau,r}(\mu) + d\E|_{P_{\tau}(\mu)} \Big(\int_{K_r^2}  \left[ \nabla_\theta \Phi_{\tau}(\theta;\cdot) \cdot (\tilde{\theta} - \theta) + M_\tau( \theta, \tilde{\theta}; \cdot)\right] d\gamma(\theta, \tilde{\theta})\Big) \\
&+ N\Big(\int_{K_r^2} M_\tau( \theta, \tilde{\theta}; \cdot) d\gamma(\theta, \tilde{\theta}) + \int_{K_r^2}\nabla_\theta \Phi_{\tau}(\theta;\cdot) \cdot (\tilde{\theta} - \theta)\,d\gamma(\theta, \tilde{\theta})\Big),\\
\end{aligned}
$$
Using standard derivation integral theorems, a bound on $M_\tau$ is available :
$$
\begin{aligned}
\Big\| \int_{K_r^2} M_\tau( \theta, \tilde{\theta}; \cdot) d\gamma(\theta, \tilde{\theta}) \Big\|^2_{H^1(\Omega)} &= \Big\| \int_{K_r^2} M_\tau( \theta, \tilde{\theta}; \cdot) d\gamma(\theta, \tilde{\theta}) \Big\|^2_{L^2(\Omega)} + \Big\| \int_{K_r^2} \nabla_x M_\tau( \theta, \tilde{\theta}; \cdot) d\gamma(\theta, \tilde{\theta}) \Big\|^2_{L^2(\Omega)}\\
& \leq   \int_{K_r^2} \|M_\tau( \theta, \tilde{\theta}; \cdot) \|^2_{L^2(\Omega)} d\gamma(\theta, \tilde{\theta})  + \int_{K_r^2} \| \nabla_x M_\tau( \theta, \tilde{\theta}; \cdot)\|^2_{L^2(\Omega)} d\gamma(\theta, \tilde{\theta})  \\
& \leq  C(r^2\tau^2 + r^2\tau^4) \int_{\Theta^2} |\tilde{\theta} - \theta|^{4} d\gamma(\theta, \tilde{\theta})\\
& \leq  C(r^4\tau^2 + r^4\tau^4) \int_{\Theta^2} |\tilde{\theta} - \theta|^{2} d\gamma(\theta, \tilde{\theta})\\
& = C(r^4\tau^2 + r^4\tau^4) c_2(\gamma),\\
\end{aligned}
$$
where we used Jensen inequality to get the first inequality and Lemma~\ref{lemmaSmoothnessGeodesic} to get the last inequality. Using Corollary~\ref{cor:boundPtau} and the uniform continuity of $d \E$, it holds :

$$
\left|d \E |_{P_\tau \mu} \left( \int_{K_r^2} M_\tau( \theta, \tilde{\theta}; \cdot) d\gamma(\theta, \tilde{\theta})  \right) \right| \leq C_{\tau,r} c_2(\gamma).
$$

Moreover, using similar calculations, it holds that
\begin{align*}
\left\|\int_{K_r^2}\nabla_\theta \Phi_{\tau}(\theta;\cdot) \cdot (\tilde{\theta} - \theta)\,d\gamma(\theta, \tilde{\theta})\right\|_{H^1(\Omega)}^2 &=
\left\| \int_{K_r^2}\nabla_\theta \Phi_{\tau}(\theta;\cdot) \cdot (\tilde{\theta} - \theta)\,d\gamma(\theta, \tilde{\theta})\right\|_{L^2(\Omega)}^2 
\\
& + \left\|\int_{K_r^2}\nabla_x\nabla_\theta \Phi_{\tau}(\theta;\cdot) \cdot (\tilde{\theta} - \theta)\,d\gamma(\theta, \tilde{\theta})\right\|_{L^2(\Omega)}^2,\\
&\leq
\int_{K_r^2}\left\|  \nabla_\theta \Phi_{\tau}(\theta;\cdot) \cdot (\tilde{\theta} - \theta)\right\|_{L^2(\Omega)}^2 \,d\gamma(\theta, \tilde{\theta})
\\
& + \int_{K_r^2}\left\|\nabla_x\nabla_\theta \Phi_{\tau}(\theta;\cdot) \cdot (\tilde{\theta} - \theta)\right\|_{L^2(\Omega)}^2\,d\gamma(\theta, \tilde{\theta}),\\
& \leq C (r^2 + r^2\tau^2) \int_{\Theta^2} |\tilde{\theta} - \theta|^2 d\gamma(\theta, \tilde{\theta})\\
& \leq C(r^2 + r^2\tau^2) c_2(\gamma).
\end{align*}

Hence, together with the previous bounds and~\eqref{boundN}, we easily obtain that there exists a constant $C_{r,\tau}>0$ such that for all $\mu, \nu \in \Prob_2(K_r)$, it holds that
\begin{equation}
\label{formulaEpdiff}
\left|\mathcal{E}_{\tau,r}(\nu) - \mathcal{E}_{\tau,r}(\mu) + d\E|_{P_\tau(\mu)} \left(\int_{K_r^2} \nabla_\theta \Phi_{\tau}(\theta;\cdot) \cdot (\tilde{\theta} - \theta) d\gamma(\theta, \tilde{\theta})\right) \right| \leq C_{r,\tau} c_2(\gamma).
\end{equation}

Now we focus on the first order term and by Fubini and standard integral derivation theorem, we obtain that:
$$
\begin{aligned}
d\E|_{P_\tau(\mu)} \left(\int_{K_r^2}  \nabla_\theta \Phi_{\tau}(\theta;\cdot) \cdot (\tilde{\theta} - \theta) d\gamma\right) &= \left\langle \nabla_x P_\tau(\mu), \nabla_x \int_{K_r^2}   \nabla_{\theta}\Phi_{\tau}(\theta;\cdot) \cdot (\tilde{\theta} - \theta) d\gamma(\theta, \tilde{\theta}) \right\rangle_{L^2(\Omega)}\\
&- \left\langle f, \int_{K_r^2}   \nabla_{\theta}\Phi_{\tau}(\theta;\cdot) \cdot (\tilde{\theta} - \theta) d\gamma(\theta, \tilde{\theta}) \right\rangle_{L^2(\Omega)}\\
&+ \int_{\Omega} P_\tau(\mu)(x)\,dx \times \int_{\Omega} \int_{K_r^2}  \nabla_{\theta}\Phi_{\tau}(\theta;x) \cdot (\tilde{\theta} - \theta) d\gamma(\theta, \tilde{\theta}) dx \\
&= \int_{K_r^2} \langle \nabla_x P_\tau(\mu), \nabla_x \nabla_{\theta}  \Phi_{\tau}(\theta;\cdot) \cdot (\tilde{\theta} - \theta) \rangle_{L^2(\Omega)}  d\gamma(\theta, \tilde{\theta}) \\
&- \int_{K_r^2} \nabla_{\theta} \langle f,  \Phi_{\tau}(\theta;\cdot)   \rangle_{L^2(\Omega)} \cdot (\tilde{\theta} - \theta) d\gamma(\theta, \tilde{\theta})\\
&+ \int_{K_r^2} \int_{\Omega} P_\tau(\mu)(x) dx \times \int_{\Omega}    \nabla_{\theta}\Phi_{\tau}(\theta;x) \cdot (\tilde{\theta} - \theta) dx d\gamma(\theta, \tilde{\theta})\\
& =  \int_{K_r^2} v_\mu(\theta) \cdot(\tilde{\theta} - \theta) d\gamma(\theta, \tilde{\theta}),
\end{aligned}
$$
where 
\begin{equation}
\label{defv}
v_\mu(\theta) := \nabla_\theta \phi_\mu(\theta) \quad \gamma-\text{almost everywhere},
\end{equation}
with
$$
\phi_\mu(\theta) :=  \langle \nabla_x P_\tau(\mu), \nabla _x \Phi_{\tau}(\theta;\cdot) \rangle_{L^2(\Omega)} - \langle f,  \Phi_{\tau}(\theta;\cdot)   \rangle_{L^2(\Omega)} + \int_{\Omega}  P_\tau(\mu)(x) dx \times \int_{\Omega} \Phi_{\tau}(\theta;x) dx.
$$
Note that~\eqref{defv} is equivalent to 
$$
v_\mu(\theta) := \nabla_\theta \phi_\mu(\theta) \quad  \mu-\text{almost everywhere},
$$
as $v_\mu$ only depends on $\theta$. 
\newline
\end{proof}

To prove a well-posedness result, some convexity is needed. More precisely, one should check that $\E_{\tau,r}$ is convex along geodesics.
\begin{proposition}
\label{propositionConvexityE}
For all $\tau, r>0$, there exists $\lambda_{\tau,r}>0$ such that for all $\mu,\nu \in \Prob_2(K_r)$ with associated geodesic $\kappa(t) := {e_t}\# \Pi$ given by Proposition~\ref{propositionMccannInterpolant}, the functional $[0,1] \ni t \mapsto \frac{d}{dt}\left(\E_{\tau,r}(\kappa(t)\right)$ is $-\lambda_{\tau,r}$-Lipschitz.

\end{proposition}

\begin{proof}

First of all, one has to check that for all $t\in [0,1]$, $\kappa(t) \in \Prob_2(K_r)$. This is a direct consequence of the fact that $\mu,\nu$ are supported in $K_r$, Remark~\ref{remarkSupportGeodesic} and that $K_r$ is convex (in the geodesic sense). 
\newline

Let $t,s \in [0,1]$ and define $\alpha(t,s) := (e_t, e_s) \# \Pi \in \Gamma(\kappa(t),\kappa(s))$. 
By~\eqref{formulaEpdiff}, it holds that
$$
\left|\mathcal{E}_{\tau,r}(\kappa(s)) - \mathcal{E}_{\tau,r}(\kappa(t)) + \int_{\Theta^2} d\E|_{P_\tau(\kappa(t))} \Big( \nabla_\theta \Phi_{\tau}(\theta;\cdot) \cdot (\tilde{\theta} - \theta) \Big) d\alpha(t,s)(\theta, \tilde{\theta}) \right| \leq C_{r,\tau} c_2(\alpha(t,s)),
$$
which reads equivalently as
\begin{align*}
&\left|\frac{\mathcal{E}_{\tau,r}(\kappa(s)) - \mathcal{E}_{\tau,r}(\kappa(t))}{s-t} - \int_{\mathfrak{P}} d\E|_{P_\tau(\kappa(t))} \Big( \nabla_\theta \Phi_{\tau}(\pi(t) ;\cdot) \cdot \Big( \frac{\pi(s) - \pi(t)}{s-t} \Big) \Big) d\Pi(\pi)\right| \\
\leq &C_{r,\tau} \frac{1}{|s-t|}\int_{\Theta^2} |\theta - \tilde{\theta}|^2 \,d\alpha(t,s)(\theta, \tilde{\theta})\\
 = &C_{r,\tau} \frac{1}{|s-t|}\int_{\mathfrak P} |\pi(t) - \pi(s)|^2 \,d\Pi(\pi)\\
 = &C_{r,\tau}|s-t| \int_{\mathfrak P} |\pi(1) - \pi(0)|^2 \,d\Pi(\pi)\\
\leq& C_{r,\tau}|s-t|, \\
\end{align*}
where the value of the constant $C_{r,\tau}$ only depends on $r$ and $\tau$.
Letting $s$ go to $t$ and using the dominated convergence theorem, one concludes that $[0,1] \ni t \mapsto \mathcal{E}_{\tau,r}(\kappa(t))$ is differentiable with derivative equal to :

$$
\begin{array}{rl}
h(t) := \frac{d}{dt}\left( \mathcal{E}_{\tau,r}(\kappa(t)) \right) = & \int_{\mathfrak P} d\E|_{P_\tau(\kappa(t))} \Big( \nabla_\theta \Phi_{\tau}(\pi(t) ;\cdot) \cdot \Big( \frac{d}{dt} \pi(t) \Big) \Big) d\Pi(\pi).
\end{array}
$$
To conclude, one has the decomposition :
\begin{equation}\label{decompositionh}
\begin{aligned}
|h(t) - h(s)| & \leq \Big|\int_{\mathfrak P} d\E|_{P_\tau(\kappa(t))} \Big( (\nabla_\theta \Phi_{\tau}(\pi(t) ;\cdot)- \nabla_\theta \Phi_{\tau}(\pi(s) ;\cdot)) \cdot  \Big( \frac{d}{dt} \pi(t) \Big) \Big) d\Pi(\pi)  \Big| \\
&+ \Big| \int_{\mathfrak P} (d\E|_{P_\tau(\kappa(t))} - d\E|_{P_\tau(\kappa(s))} ) \Big( \nabla_\theta \Phi_{\tau}(\pi(s) ;\cdot) \cdot \Big( \frac{d}{dt} \pi(t) \Big) \Big) d\Pi(\pi)  \Big|\\
&+ \Big| \int_{\mathfrak P} d\E|_{P_\tau(\kappa(s))}  \Big( \nabla_\theta \Phi_{\tau}(\pi(s) ;\cdot) \cdot \Big( \frac{d}{dt} \pi(t) - \frac{d}{dt} \pi(s) \Big) \Big) d\Pi(\pi)  \Big|.
\end{aligned}
\end{equation}

Recalling~\eqref{decompositionh}, denoting $\alpha := (e_0, e_1)\# \Pi$ and using the previous estimates, we obtain that, for all $t,s\in [0,1]$,
$$
\begin{aligned}
|h(t) - h(s)| &\leq C_{r,\tau} \Big( \|  P_\tau(\kappa(t)) \|_{H^1(\Omega)} \int_{\mathfrak P}  |\pi(t) - \pi(s)| \Big| \frac{d}{dt} \pi(t) \Big| d\Pi(\pi) \\
&+  \|  P_\tau(\kappa(t)) - P_\tau(\kappa(s)) \|_{H^1(\Omega)} \int_{\mathfrak P}  |\pi(s)| \Big| \frac{d}{dt} \pi(t) \Big| d\Pi(\pi) \\
&+\|  P_\tau(\kappa(s)) \|_{H^1(\Omega)}\int_{\mathfrak P}  |\pi(s)| \Big| \frac{d}{dt} \pi(t) - \frac{d}{dt} \pi(s) \Big| d\Pi(\pi) \Big) \\
&\leq C_{r,\tau} \Big( |t-s| \|  P_\tau(\kappa(t)) \|_{H^1(\Omega)} \int_{\mathfrak P}  |\pi(1) - \pi(0)|^2  d\Pi(\pi) \\
&+  \|  P_\tau(\kappa(t)) - P_\tau(\kappa(s)) \|_{H^1(\Omega)} \int_{\mathfrak P}  \sup_{u\in[0,1]} |\pi(u)| |\pi(1) - \pi(0)| d\Pi(\pi) \\
&+ |t-s|\|  P_\tau(\kappa(s)) \|_{H^1(\Omega)}\int_{\mathfrak P}  \sup_{u\in[0,1]} |\pi(u)| \sup_{u\in[0,1]} \left|\frac{d^2 \pi(u)}{dt^2}\right|  d\Pi(\pi) \Big)\\
& \leq C_{r,\tau} \left( |t-s| \left(\sqrt{\int_{\Theta^2} |\theta|^2 d\kappa(t)(\theta)}
+ \sqrt{\int_{\Theta^2} |\theta|^2 d\kappa(s)(\theta)}\right)  (1+c_2(\alpha)) + W_2(\kappa(t),\kappa(s)) c_2(\alpha) \right)
\end{aligned}
$$
where we have used Lemma~\ref{lemmaSmoothnessGeodesic} to get the second inequality and the fact that $\sup_{u\in[0,1]} \left|\frac{d^2 \pi(u)}{dt^2}\right|$ is uniformly bounded (since the curvature of $\Theta$ is bounded) to get the last one. 
We also have the following estimates: 
\begin{itemize}
\item By Remark~\ref{remarkSupportGeodesic} and the convexity of $K_r$ (in the geodesic sense), for all $0 \leq t \leq 1$ :
$$
\int_{\Theta} |\theta|^2 d\kappa(t)(\theta)\leq C(1+r^2).
$$

\item Moreover,
$$
\begin{aligned}
W_2^2(\kappa(t),\kappa(s)) & \leq  \int_{\Theta^2} d(\theta, \tilde{\theta})^2 d \alpha(t,s)(\theta, \tilde{\theta})\\
& \leq  \int_\Gamma d(\pi(t), \pi(s))^2 d \Pi(\pi) \\
& =  |t-s| \int_\Gamma d(\pi(1), \pi(0))^2 d \Pi(\pi) \\
& =  |t-s| \int_{\Theta^2} d(\theta, \tilde{\theta})^2 d \alpha(\theta, \tilde{\theta}).
\end{aligned}
$$

\end{itemize}
This allows us to conclude that :

$$
|h(t) - h(s)| \leq C_{r,\tau} (1+c_2(\alpha)) |t-s|.
$$
As the measure $\alpha$ is supported in $K_r^2$, we get :
$$
|h(t) - h(s)| \leq \lambda_{\tau,r} |t-s|.
$$
for some $\lambda_{\tau, r} >0$, which yields the desired result.
\end{proof}

The characterization of the velocity field allows to get a bound on its amplitude. This is given by the next corollary which will be useful later in the paper.
\begin{corollary}\label{corollaryBoundVelocity}
There exists a constant $C_\tau >0$ such that for all $r>0$, all $\mu \in \Prob_2(K_r)$ and $\theta \in \Theta$: 
$$
|v_\mu(\theta)| \leq C_\tau r |\theta|.
$$
\end{corollary}

\begin{proof}
This can be proved combining~\eqref{BoundPbarronmu},~\eqref{gradtheta} and~\eqref{gradthetax}. The rest is just basic computations and left to the reader.
\end{proof}

An important consequence of Proposition~\ref{propositionConvexityE} is that $\E_{\tau,r}$ is $(-\lambda_{\tau,r})$-convex along geodesics. Now we are able to prove Theorem~\ref{theoremWellPosednessCondSupport}.

%
%

%

\begin{proof}[Proof of Theorem~\ref{theoremWellPosednessCondSupport}]


The functional $\E_{\tau,r}$ is lower semicontinuous by Remark~\ref{remarkLowerSemicontinuity} and it is $(-\lambda_{\tau,r})$-convex along generalized geodesics. Moreover, the space $\Theta$ has a curvature bounded from below which ensures that it is an Alexandrov space of curvature bounded from below. We apply~\cite[Theorem 5.9, 5.11]{Ohta2009} to get the existence and the uniqueness of a gradient curve $\mu^r: \mathbb{R}_+ \to \Prob_2(K_r)$ in the sense of~\cite[Definition 5.8]{Ohta2009}. Being a gradient curve, it is also a curve of maximal slope in the sense of~\cite[Definition 1.3.2]{AmbrosioSavareBook}. Note that in~\cite{Ohta2009}, the space on which the probability measures are defined (here this is $\Theta$) is supposed to be compact. This is not a problem here since the domain of the functional $\E_{\tau,r}$ is reduced to probability measures whose support is included in $K_r$ which is compact and geodesically convex.

The existence of the vector field $v^r_t$ for almost all $t\geq 0$ is given by the absolute continuity of the curve $[0,1]\ni t \mapsto \mu^r(t)$ (because it is a gradient curve) and by~\cite[Proposition 2.5]{Erbar2010}.

\end{proof}

The work is not finished here since we do not have any knowledge about the velocity field $v^r_t$ and the well-posedness result is proved only for $\E_{\tau,r}$ with $r < \infty $. In the following sections, we prove that this velocity field can be related to $v_{\mu^r(t)}$ and use a bootstrap argument to prove an existence result for the gradient curve of $\E_{\tau,+\infty}$.

\subsubsection{Identification of the vector field $v^r_t$}

In the following, we denote by $T\Theta$ the tangent bundle of $\Theta$, i.e. 
$$
T\Theta := \bigcup_{\theta\in \Theta} \{\theta\} \times T_\theta \Theta, 
$$
where $T_\theta \Theta$ is the tangent space to $\Theta$ at $\theta$. It is easy to check that for all $\theta = (c,a,w,b) \in \Theta$, it holds that $T_\theta \Theta = \mathbb{R} \times \mathbb{R} \times {\rm Span}\{w\}^\perp \times \mathbb{R}$, where ${\rm Span}\{w\}^\perp$ is the subspace of $\mathbb{R}^d$ containing all $d$-dimensional vectors orthogonal to $w$.

We also introduce the operators $G$ and $S_h$ for $0 < h \leq 1$ as follows :

$$
G := 
\left\{
\begin{array}{rcl}
\mathfrak P & \rightarrow & T \Theta\\
\pi & \mapsto & (\pi(0), \dot{\pi}(0))
\end{array}
\right.
$$
and 

$$
S_h := 
\left\{
\begin{array}{rcl}
T \Theta & \rightarrow & T \Theta\\
(\theta, v) & \mapsto & \left(\theta, \frac{v}{h}\right).
\end{array}
\right.
$$

The next lemma concerns the local behaviour of couplings along a curve of maximal slope $\mu^r: \mathbb{R}_+ \to \Prob_2(K_r)$. In the following, for any $\mu, \nu\in \Prob_2(\Theta)$, we denote by $\Gamma_o(\mu,\nu)$ the set of optimal transport plans between $\mu$ and $\nu$ in the sense of the quadratic cost. In other words, for all $\gamma \in \Gamma_o(\mu,\nu)$, it holds that $W_2^2(\mu,\nu) = \int_{\Theta \times \Theta} |\theta - \widetilde{\theta}|^2\,d\gamma(\theta, \widetilde{\theta})$.
\begin{lemma}\label{lemmaDiff}
Let $\mu^r: \mathbb{R}_+ \to \Prob_2(K_r)$ be a solution to~\eqref{wellPosednessEr} and for all $0< h\leq 1$, let $\Pi_h\in \Prob_2(\mathfrak P)$ such that $\gamma_h  := (e_0,e_1)\#\Pi_h \in \Gamma_o(\mu^r(t),\mu^r(t+h))$ (i.e. satisfying the condition of Proposition~\ref{propositionMccannInterpolant} with $\mu = \mu^r(t)$ and $\nu = \mu^r(t+h)$). Then, for almost all $t\geq 0$, it holds that
$$
\lim_{h \rightarrow 0} (S_h \circ G)\# \Pi_h = (i \times v^r_t) \# \mu^r(t) \text{ in } \Prob_2(T\Theta),
$$
where $(v^r_t)_{t\geq 0}$ is given by Theorem~\ref{theoremWellPosednessCondSupport},  and $i:\Theta \to \Theta$ is the identity map.

Moreover, 
$$
\lim_{h \rightarrow 0} \frac{W^2_2(\mu^r(t+h), \exp(h v^r_t){\#} \mu^r(t))}{h^2} = 0,
$$
where $\exp(h v^r_t): \Theta \ni \theta \mapsto \exp_\theta(h v^r_t(\theta))$.
\end{lemma}

\begin{proof}
Let $\phi$ be in $C^\infty_{c}(\Theta)$. The continuity equation gives :

$$
\int_{\R_+} \eta^\prime(t) \int_\Theta \phi \,d \mu^r(t) dt= - \int_{\R_+} \eta(t) \int_\Theta \nabla_\theta \phi \cdot v_t \,d \mu^r(t) dt
$$
for $\eta$ smooth compactly supported in $\R_+$. Taking $\eta$ as an approximation of the characteristic function of $[t,t+h]$, owing to the fact that $\mu^r$ is locally Lipschitz and passing to the limit, one gets :

$$
\int_{\Theta}\phi\,d\mu^r(t) - \int_\Theta \phi \,d\mu^r(t+h)= -\int_{t}^{t+h} \int_\Theta \nabla_\theta \phi \cdot v^r_t \,d \mu^r(t) dt.
$$
Passing to the limit as $h$ goes to $0$, one gets the differentiability almost everywhere of $\mathbb{R}_+ \ni t \mapsto \int_\Theta \phi \, d\mu^r(t)$ and :

$$
\lim_{h \rightarrow 0} \frac{\int_\Theta \phi\,d\mu^r(t+h) - \int_\Theta \phi \,d\mu^r(t)}{h} = \int_\Theta \nabla_\theta \phi \cdot v^r_t \,d \mu^r(t).
$$
For all $0< h\leq 1$, let us introduce $\nu_h := (S_h \circ G)\# \Pi_h$ and let $\nu_0$ be an accumulation point of $(\nu_h)_{0<h\leq 1}$ with respect to the narrow convergence on $\Prob_2(T\Theta)$.

Then, it holds that 
$$
\begin{aligned}
\frac{\int_\Theta \phi\,d\mu^r(t+h) - \int_\Theta \phi \,d\mu^r(t)}{h}  &= \frac{1}{h} \int_{\Theta^2} (\phi(\tilde{\theta}) - \phi(\theta)) \,d \gamma_h(\theta, \tilde{\theta})\\
&= \frac{1}{h} \int_{\mathfrak P} (\phi(\pi(1)) - \phi(\pi(0))) \,d \Pi_h(\pi)\\
&= \frac{1}{h} \int_{T \Theta} (\phi(\exp_\theta(v)) - \phi(\theta)) \,d {G}\# \Pi_h(\theta, v)\\
&= \frac{1}{h} \int_{T \Theta} (\phi(\exp_\theta(h v)) - \phi(\theta)) \,d {(S_h \circ G)}\# \Pi_h(\theta, v)\\
&=\int_{T \Theta} \nabla_\theta \phi(\theta) \cdot v   \,d \nu_h(\theta, v) \\
 &+ \int_{T \Theta}  R_h(\theta,v)  \,d \nu_h(\theta, v)\\
&\mathop{\longrightarrow}_{h\to 0} \int_{T \Theta} \nabla_\theta \phi(\theta) \cdot v d\nu_0(\theta,v),
\end{aligned}
$$
where $R_h(\theta,v) := \frac{\phi(\exp_\theta(h v)) - \phi(\theta)}{h} - \nabla_\theta \phi(\theta) \cdot v$ is bounded by $C(\phi) |v|^2 h$ ($\phi \in C_c^\infty(\Theta)$ and the euclidean curvature in $\Theta$ is uniformly bounded; see~\cite[Chapter 8]{LeeIntroCurv} for the definition of euclidean curvature). Actually, to get the last limit, we need the following arguments detailed below :

\begin{itemize}

\item For the first term, $\nabla \phi(\theta) \cdot v$ is quadratic in $(\theta,v)$ and consequently the passage to the limit is allowed.


\item For the second one, 

$$
\begin{array}{rl}
\int_{T \Theta}  |R_h(\theta,v)|  d \nu_h(\theta, v) \leq& C(\phi)h \int_{T \Theta}  |v|^2  d \nu_h(\theta, v) \\
= & C(\phi) h \frac{W^2_2(\mu^r(t), \mu^r(t+h))}{h^2} 
\end{array}
$$
and using again the local Lipschitz property, we can pass to the limit which is zero.

\end{itemize}
 
As a consequence,

$$
\int_{T \Theta} \nabla_\theta \phi(\theta) \cdot v d \nu_0(\theta, v) = \int_\Theta \nabla_\theta \phi (\theta) \cdot v^r_t(\theta) \,d \mu^r(t)(\theta)
$$
which is no more than (by disintegration) :

$$
\int_{\Theta} \nabla_\theta \phi(\theta) \cdot \int_{T_\theta \Theta}v \,d \nu_{0,\theta}(v) \,d \mu^r(t)(\theta) = \int_\Theta \nabla_\theta \phi(\theta) \cdot v^r_t(\theta) \,d \mu^r(t)(\theta).
$$
Noting $\tilde{v_t}(\theta) := \int_{T_\theta \Theta}v \,d \nu_{0,\theta}(v)$, the last equation is equivalent to :

$$
{\rm div}((\tilde{v_t} - v^r_t) \mu^r(t)) = 0.
$$
In addition, as $T\Theta \ni (\theta, v) \mapsto |v|^2$ is positive and lower semicontinuous and as for almost all $t\geq 0$ we have that $\lim_{h \rightarrow 0} \frac{W_2(\mu^r(t),\mu^r(t+h))}{h}
= |(\mu^r)^\prime|(t)$ (as $\mu^r$ is locally Lipschitz): 

\begin{equation}\label{convMoments}
\begin{aligned}
\int_{\Theta} \int_{T_\theta \Theta} |v|^2 \,d \nu_{0,\theta}(v) \,d \mu^r(t)(\theta) &\leq \liminf_{h \rightarrow 0} \int_{T \Theta} |v|^2 \,d \nu_{h}(\theta, v) \\
&= \liminf_{h \rightarrow 0} \frac{1}{h^2} \int_{T \Theta} |v|^2 \,d {G}\# \Pi_h (\theta, v)\\
&= \liminf_{h \rightarrow 0} \frac{1}{h^2} \int_{\mathfrak P} |\dot{\pi}(0)|^2 \,d \Pi_h (\pi)\\
&= \liminf_{h \rightarrow 0} \frac{1}{h^2} \int_{\Theta^2} d(\theta, \tilde{\theta})^2 \,d\gamma_{h}(\theta, \tilde{\theta})\\
&=  \liminf_{h \rightarrow 0} \frac{W_2^2(\mu^r(t),\mu^r(t+h))}{h^2}\\
&= |(\mu^r)^\prime|^2(t).
\end{aligned}
\end{equation}
As a consequence and by Jensen inequality,

\begin{equation}\label{JensenIneq}
\| \tilde{v}_t \|^2_{L^2(\Theta; d\mu^r(t))} \leq \int_{\Theta} \int_{T_\theta\Theta} |v|^2 \,d \nu_{0,\theta}(v) \,d \mu^r(t)(\theta) \leq |(\mu^r)^\prime|^2(t) =  \| v^r_t \|^2_{L^2(\Theta; d\mu^r(t))}.
\end{equation}
By~\cite[Lemma 2.4]{Erbar2010}, one gets $\tilde{v}_t = v^r_t$. Reconsidering~\eqref{JensenIneq}, one gets the equality case in Jensen inequality \textit{ie} :

$$
\int_{\Theta} \left| \tilde{v}_t(\theta) \right|^2 \,d \mu^r(t)(\theta) = \int_{\Theta} \int_{T_\theta \Theta} |v|^2 d \nu_{h,\theta}(v) \,d \mu^r(t)(\theta),
$$
and as a consequence $\nu_{0,\theta} = \delta_{v^r_t(\theta)}$, $\mu^r(t)$-almost everywhere in $\Theta$. In addition,
$$
\lim_{h \rightarrow 0} (S_h \circ G)\# \Pi_h = (i \times v^r_t){\#} \mu^r(t),
$$
in the sense of the narrow convergence. The convergence of the $v$ moment is given by~\eqref{convMoments}-\eqref{JensenIneq} where inequalities can be replaced by equalities (as $\tilde{v}_t = v^r_t$) and the $\liminf$ can be replaced by a $\lim$ as  $\lim_{h \rightarrow 0} \frac{W_2(\mu^r(t),\mu^r(t+h))}{h}
= |(\mu^r)^\prime|(t)$ exists : 
\begin{equation}\label{convMoment1}
\int_{\Theta} \int_{T_\theta \Theta} |v|^2 d \nu_{0,\theta}(v) \,d \mu^r(t)(\theta) = \lim_{h \rightarrow 0} \int_{T \Theta} |v|^2 d \nu_{h}(\theta, v).
\end{equation}

For the $\theta$ moment, it is more obvious as for all $0 < h \leq 1$ :

$$
\int_{T \Theta} |\theta|^2 \,d \nu_{h}(\theta, v) = \int_{\Theta} |\theta|^2 \,d \mu^r(t)(\theta)
$$
and 
$$
\int_{T \Theta} |\theta|^2 \,d \nu_{0}(\theta, v) = \int_{T \Theta} |\theta|^2 d (i \times v^r_t){\#} \mu^r(t)(\theta) =  \int_{\Theta} |\theta|^2 \,d \mu^r(t)(\theta).
$$
Consequently,

\begin{equation}\label{convMoment2}
\int_{T \Theta} |\theta|^2 d \nu_{0}(\theta, v)  = \lim_{h \rightarrow 0} \int_{T \Theta} |\theta|^2 d \nu_{h}(\theta, v).
\end{equation}

With~\eqref{convMoment1}-\eqref{convMoment2}, the convergence of moments is asserted. The narrow convergence combined with the convergence of moments gives the convergence in $\Prob_2(\Theta)$ and the proof of the first part of the lemma is finished.
\newline

For the second part, it holds that $(\exp(h v^r_t) \times i){\#}\gamma_h$ belongs to $\Gamma(\exp(h v^r_t){\#}\mu^r(t), \mu^r(t+h))$. Hence,

$$
\begin{aligned}
\frac{W^2_2(\mu^r(t+h), \exp(h v^r_t){\#} \mu^r(t))}{h^2} &\leq \frac{1}{h^2} \int_{\Theta^2} \,d(\theta, \tilde{\theta})^2 d(\exp(h v^r_t) \times i){\#}\gamma_h(\theta, \tilde{\theta})\\
& \leq  \frac{1}{h^2} \int_{\Theta^2} d(\exp_\theta(h v^r_t(\theta)) , \tilde{\theta})^2 \,d\gamma_h(\theta, \tilde{\theta})\\
& \leq  \frac{1}{h^2} \int_{T\Theta} d(\exp_\theta(h v^r_t(\theta)) , \exp_\theta(h v))^2 \,d\nu_h(\theta, v) \\
& \leq  C \int_{T\Theta} |v^r_t(\theta) - v|^2 \,d\nu_h(\theta, v)\\
&\mathop{\longrightarrow}_{h \to 0}  0,
\end{aligned}
$$
where we have used the boundedness of the euclidean curvature of the manifold $\Theta$ in the last inequality and the fact that $\nu_h \rightarrow (i \times v^r_t){\#} \mu^r(t)$, which was proved earlier. Hence the desired result.
\end{proof}

We now introduce the projection operator on the manifold $\Theta$ :

\begin{definition}\label{def:P}
For all $\theta$ in $\Theta$, the orthogonal projection on the tangent space of $ \Theta$ is given by the operator $\bold{P}_\theta : \R^{d+3} \rightarrow T_\theta \Theta$. The operator $\bold{P}: L^1_{\rm loc}(\Theta; \R^{d+3}) \to L^1_{\rm loc}(\Theta; \R^{d+3})$ denotes the corresponding projection on vector fields, \textit{i.e.} for all $X\in L^1_{\rm loc}(\Theta; \R^{d+3})$, $(\bold{P} X)(\theta) := \bold{P}_\theta X(\theta)$ for almost all $\theta\in \Theta$.

\end{definition}

Now we are able to identify the velocity field given in Theorem~\ref{theoremWellPosednessCondSupport} under a support hypothesis.

\begin{proposition}\label{velocityProp}
Let $t\geq 0$. If there exists $\delta>0$ such that ${\rm Supp}(\mu^r(t)) \subset K_{r-\delta}$, then the velocity field $v^r_t$ in~\eqref{wellPosednessEr} is equal to $- \bold{P} v_{\mu^r(t)}$ $\mu^r(t)$-almost everywhere.
\end{proposition}

\begin{proof}
On the one hand, for $\gamma_h := (e_0, e_1)\# \Pi_h \in \Gamma_o(\mu^r(t), \mu^r(t+h))$, by Proposition~\ref{lemmaDiffEp} and the fact that for all $t\geq 0$, $\mu^r(t) \in \Prob_2(K_r)$ : 

$$
\left|\mathcal{E}_{\tau,r}(\mu^r(t+h)) - \mathcal{E}_{\tau,r}(\mu^r(t)) - \int_{\Theta^2} v_{\mu^r(t)}(\theta) \cdot(\tilde{\theta} - \theta)\, d\gamma_h(\theta, \tilde{\theta})\right| \leq  C_{r,\tau}W_2(\mu^r(t),\mu^r(t+h))^2,
$$
which is equivalent to 
$$
\left|\frac{\mathcal{E}_{\tau,r}(\mu^r_{t+h}) - \mathcal{E}_{\tau,r}(\mu^r(t))}{h} - \int_{T\Theta} v_{\mu^r(t)}(\theta) \cdot \frac{\exp_\theta(h v) - \theta}{h} \,d (S_h \circ G)\# \Pi_h(\theta, v) \right| \leq C_{r,\tau} \frac{1}{h} W_2(\mu^r(t),\mu^r(t+h))^2.
$$
Then, one can use the decomposition :

$$
\begin{aligned}
\int_{T\Theta} v_{\mu^r(t)}(\theta) \cdot \frac{\exp_\theta(h v) - \theta}{h}\, d(S_h \circ G)\# \Pi_h(\theta, v) & =  \int_{T\Theta} v_{\mu^r(t)}(\theta) \cdot v \,d(S_h \circ G)\# \Pi_h(\theta, v) \\
& + \int_{T\Theta} v_{\mu^r(t)}(\theta) \cdot R_h(\theta,v) \,d(S_h \circ G)\# \Pi_h(\theta, v),
\end{aligned}
$$
where $R_h(\theta,v) := \frac{\exp_\theta(h v) - \theta}{h} - v$ is bounded by $C h |v|^2$ due to the uniform boundedness of euclidean curvature in $\Theta$. Passing to the limit as $h$ goes to zero and using Lemma~\ref{lemmaDiff}, one gets the differentiability of $\mathbb{R}_+ \ni t \rightarrow \E_{\tau,r}(\mu^r(t))$ almost everywhere and for almost all $t\geq 0$ :

$$
\frac{d}{dt}\left[\E_{\tau,r}(\mu^r(t))\right]= \int_{\Theta} v_{\mu^r(t)}(\theta) \cdot v^r_t(\theta) \,d\mu^r(t)(\theta).
$$
Note that to pass to the limit to obtain the last equation, we need the two  following points :

\begin{itemize}
\item First, $v \cdot v_{\mu^r(t)}(\theta)$ is at most quadratic in $(\theta,v)$ which is given by Corollary~\ref{corollaryBoundVelocity}.

\item Second, it holds that $|v_{\mu^r(t)}(\theta) \cdot R_h(\theta,v)| \leq C r |\theta| h |v|^2$ by Corollary~\ref{corollaryBoundVelocity} and consequently :

$$
\begin{array}{rl}
\left| \int_{T\Theta} v_{\mu^r(t)}(\theta) \cdot R_h(\theta,v) \,d(S_h \circ G)\# \Pi_h(\theta, v) \right| \leq& C_r h \int_{T\Theta} |\theta| |v|^2 \,d(S_h \circ G)\# \Pi_h(\theta, v)\\
\leq& C_r h \int_{T\Theta}  |v|^2 \,d(S_h \circ G)\# \Pi_h(\theta, v) \\
\leq& C_r h \frac{W_2(\mu^r(t), \mu^r(t+h))^2}{h^2}
\end{array}
$$
where we used the fact that $\Pi_h$ is supported in $K_r$ in its first variable to get the second inequality. The last term converges to zero since $(\mu_r(t))_t$ is local Lipschitz.

\end{itemize}

Next as $\bold{P} v^r_t = v^r_t$, it holds that: 
\begin{equation}
\label{limitdEdt}
\frac{d}{dt}\left[\E_{\tau,r}(\mu^r(t))\right] = \int_{\Theta^2} \bold{P} v_{\mu^r(t)}(\theta) \cdot v^r_t(\theta) \,d \mu^r(t)(\theta).
\end{equation}

On the other hand, consider the curve $\tilde{\mu_h}: \mathbb{R}_+\to \Prob_2(\Theta)$ satisfying :

$$
\forall t\geq 0, \quad \tilde{\mu}_h(t) := \exp(- h \bold{P} v_{\mu^r(t)} )\# \mu^r(t).
$$

As ${\rm Supp}(\mu^r(t)) \subset K_{r - \delta}$, there exists a small time interval around zero such that $\tilde{\mu}_h(t)$ is in $\Prob_2(K_r)$ for $h>0$ small enough. So, with $\gamma_h := (i \times \exp(-h \bold{P} v_{\mu^r(t)}))\# \mu^r(t) \in \Gamma(\mu^r(t),\tilde{\mu}_h(t))$,
$$
\left|\mathcal{E}_{\tau,r}(\tilde{\mu}_{h}(t)) - \mathcal{E}_{\tau,r}(\mu^r(t)) - \int_{\Theta^2} \bold{P} v_{\mu^r(t)}(\theta) \cdot(\tilde{\theta} - \theta) \,d\gamma_h(\theta, \tilde{\theta})\right| \leq C_{r,\tau} W^2_2(\mu^r(t),\tilde{\mu}_{h}(t))
$$
and it holds that
$$
\int_{\Theta^2} \bold{P} v_{\mu^r(t)}(\theta) \cdot(\tilde{\theta} - \theta) \,d\gamma_h(\theta, \tilde{\theta})= h \int_{\Theta^2} \bold{P} v_{\mu^r(t)}(\theta) \cdot \frac{\exp_\theta(-h \bold{P} v_{\mu^r(t)}(\theta)) - \theta}{h} \,d\mu^r(t)(\theta).
$$
Hence,  

$$
\frac{\mathcal{E}_{\tau,r}(\tilde{\mu}_{h}(t)) - \mathcal{E}_{\tau,r}(\mu^r(t))}{W_2(\tilde{\mu}_h(t), \mu^r(t)) } = \frac{h}{W_2(\tilde{\mu}_h(t), \mu^r(t))} \int_{\Theta^2} \bold{P} v_{\mu^r(t)}(\theta) \cdot \frac{\exp_\theta(-h \bold{P} v^\tau_{\mu^r(t)}(\theta)) - \theta}{h} \, d\mu^r(t)(\theta) + o_h(1)
$$
and getting the limsup as $h$ goes to zero (proceeding in the similar way as above to get the limit of the first term on the right hand side) and owing to the fact that $\limsup_{h \rightarrow 0} \frac{W_2(\tilde{\mu}_h(t), \mu^r(t))}{h} \leq \| \bold{P} v_{\mu^r(t)} \|_{{L^2(\Theta;d\mu^r(t))}}$, we obtain that

\begin{equation}\label{estimateLocalSlope}
|\nabla^- \E_{\tau,r} |(\mu^r(t)) \geq \| \bold{P} v_{\mu^r(t)} \|_{L^2(\Theta;d\mu^r(t))}.
\end{equation}
As $\mu^r$ is a curve of maximal slope with respect to the upper gradient $|\nabla^- \E_{\tau,r}|$ of $\E_{\tau,r}$, one has :

$$
\begin{aligned}
\frac{d}{dt}\left[\E_{\tau,r}(\mu^r(t))\right] &= \int_{\Theta} \bold{P} v_{\mu^r(t)}(\theta) \cdot v^r_t(\theta) \,d \mu^r(t)(\theta) \leq - \frac{1}{2} \| v^r_{t} \|_{L^2(\Theta;d\mu^r(t))} - \frac{1}{2} |\nabla^- \E_{\tau,r}|^2(\mu^r(t))\\
& \leq  - \frac{1}{2} \| v^r_{t} \|^2_{L^2(\Theta;d\mu^r(t))} - \frac{1}{2} \| \bold{P} v_{\mu^r(t)} \|^2_{L^2(\Theta;d\mu^r(t))}
\end{aligned}
$$
where we have used~\eqref{estimateLocalSlope}. As a consequence,

$$
\int_{\Theta} \left( \frac{1}{2} (\bold{P} v_{\mu^r(t)})^2(\theta) + \frac{1}{2} |v^r_{t}(\theta)|^2 - \bold{P} v_{\mu^r(t)}(\theta) \cdot v^r_t(\theta) \right) \, d \mu^r(t)(\theta) \leq 0
$$
and 
$$
v^r_t = - \bold{P} v_{\mu^r(t)} \quad \mu^r(t)\text{-a.e}.
$$
\end{proof}

The identification of the velocity field when the support condition is satisfied allows to give an explicit formula for the gradient curve. It is given by the characteristics :

\begin{proposition}
\label{propositionCharacteristics}

Let $\chi^r: \mathbb{R}_+ \times \Theta \to \Theta$ be the flow associated to the velocity field $- \bold{P} v_{\mu^r(t)}$ :
$$
\left\{
\begin{aligned}
\partial_t \chi^r(t) &= - \bold{P} v_{\mu^r(t)}\\
\chi^r(0;\theta) &= \theta.
\end{aligned}
\right.
$$
Then $\chi^r$ is uniquely defined, continuous, and for all $t\geq 0$, $\chi^r(t)$ is Lipschitz on $K_r$. Moreover, as long as ${\rm Supp}(\mu^r(t)) \subset K_{r-\delta}$ for some $\delta >0$ :

$$
\mu^r(t) = {\chi^r(t)}\# \mu_0.
$$
\end{proposition}

\begin{proof}
This is a direct consequence of the fact that $v^r_t =  - \bold{P} v_{\mu^r(t)} = - \bold{P} \nabla_\theta \phi_{\mu^r(t)}$ is $C^\infty$.
\end{proof}

Next lemma relates the curve $[0,1] \ni h \mapsto \exp(h v^r_t){\#} \mu^r(t)$ with $\nabla_- \E_{\tau,r} (\mu^r(t))$. This will be useful later to prove that the velocity field characterizes the gradient curve. 

\begin{lemma}\label{lemmaIdentificationGradient}
For all $\mu \in \Prob_2(\Theta)$ with ${\rm Supp}(\mu) \subset K_{r-\delta}$ for some $\delta>0$, the map $\nu: [0,1] \ni h \mapsto \exp(-h \bold{P} v_\mu / \| \bold{P} v_\mu \|_{L^2(\Theta; d\mu)})\# \mu $ is differentiable at $h = 0$. Moreover, it holds that
$$
\nu^\prime(0) = \nabla_- \E_{\tau,r} (\mu)/ |\nabla_- \E_{\tau, r}|(\mu).
$$
\end{lemma}

\begin{proof}
First, we claim that $|\nabla_- \E_{\tau,r}(\mu)| = \| \bold{P} v_\mu(\theta) \|_{L^2(\Theta; d\mu)}$. In order to prove it, take an arbitrary unit speed geodesic $[0,1] \ni s \mapsto (e_s)\# \Pi $ starting at $\mu$ for which there exists a time interval around zero such that $(e_s)\# \Pi$ belongs to $\Prob_2(K_r)$. As a consequence, one can write for all $s >0$ sufficiently small :
$$
\left|\E_{\tau,r}((e_{ s})\# \Pi) - \E_{\tau,r}(\mu) + \int_{\Theta^2} v_{\mu}(\theta) \cdot (\tilde{\theta} - \theta) \, d (e_0, e_{ s})\# \Pi(\theta, \tilde{\theta}) \right| \leq C_{r,\tau} W_2^2(\mu, (e_{ s})\# \Pi).
$$
with
$$
\int_{\Theta^2} v_{\mu}(\theta) \cdot (\tilde{\theta} - \theta) \, d (e_0, e_{ s})\# \Pi(\theta, \tilde{\theta}) =  \int_{T\Theta} v_{\mu}(\theta) \cdot (\exp_\theta( sv) - \theta) \,d G\# \Pi(\theta, v).
$$
Dividing by $s$ and passing to the limit as $s$ goes to zero, one obtains : 

$$
\frac{d}{ds}\left[\E_{\tau,r}((e_{ s})\# \Pi)\right] =  \int_{T\Theta} v_{\mu}(\theta) \cdot v \,d G\# \Pi(\theta, v).
$$
Note that, to get the last equation, we need to prove that for all $s$ sufficiently small the function $\eta(s): T\Theta \ni (\theta,v) \mapsto v_\mu(\theta) \cdot \frac{\exp_\theta( sv) - \theta}{s}$ is uniformly integrable with respect to $G\# \Pi$. In fact, this is given by Corollary~\ref{corollaryBoundVelocity} and the uniform curvature bound on $\Theta$ giving $|\eta(s)|(\theta, v) \leq C s r |\theta| |v|^2$. As the term $C r |\theta| |v|^2$ is integrable with respect to the measure $G\# \Pi$ (recall that it has finite second-order moments and is supported in $K_r$ in the $\theta$ variable), we have the desired uniform integrability property.

Moreover, by Cauchy-Schwartz inequality:

$$
\begin{aligned}
\frac{d}{ds}\left[\E_{\tau,r}((e_{ s})\# \Pi)\right] &\geq -  \| \bold{P} v_{\mu}\|_{L^2(\Theta; d\mu)} \sqrt{\int_{T\Theta} v^2 d G\# \Pi(\theta, v)} \\
& =  -  \| \bold{P} v_{\mu}\|_{L^2(\Theta; d\mu)},
\end{aligned}
$$
where the last equality comes from :
$$
\begin{aligned}
\int_{T\Theta} v^2 \, d G\# \Pi(\theta, v) &= \int_{\mathfrak P} \dot{\pi}(0)^2 \,d \Pi(\pi)\\
& =  \int_{\mathfrak P} d(\pi(0), \pi(1))^2 \, d \Pi(\pi)\\
& =  W_2^2((e_0)\# \Pi, (e_1)\# \Pi)\\
& =  1.
\end{aligned}
$$
The last equality is derived from the fact that $[0,1] \ni s \mapsto (e_s)\# \Pi$ is a unit speed geodesic. To conclude, we have proved that for all unit speed geodesic $(\alpha, 1) \in C_\mu(\Prob_2(K_r))$

$$
D_\mu \E_{\tau,r}((\alpha,1)) \geq - \| \bold{P} v_{\mu}\|_{L^2(\Theta; d\mu)}
$$
which by~\cite[Lemma 4.3]{Ohta2009}, asserts that :

\begin{equation}\label{boundNablaEP1}
|\nabla_- \E_{\tau,r}|(\mu) \leq \| \bold{P} v_{\mu}\|_{L^2(\Theta; d\mu)}.
\end{equation}

Aside that, let $h > 0$ :
$$
\begin{aligned}
W^2_2(\nu(h), \nu(0)) & \leq  \int_{\Theta} d^2\left(\exp_\theta\left(-h \bold{P} v_\mu(\theta) / \| \bold{P} v^\tau_\mu \|_{L^2(\Theta; d\mu)} \right), \theta\right) \,d\mu(\theta) \\
& \leq  h^2 \int_{\Theta} d^2\left(\exp_\theta\left(-\bold{P} v_\mu(\theta)/ \| \bold{P} v_\mu \|_{L^2(\Theta; d\mu)}\right), \theta\right) \, d\mu(\theta) \\
& =  h^2,\\
\end{aligned}
$$
and 
\begin{equation}\label{nuLipschitz}
\limsup_{h \rightarrow 0} \frac{W_2(\nu({h}), \nu_{0})}{h} \leq 1.
\end{equation}

Moreover as ${\rm Supp}(\mu) \subset K_{r-\delta}$, $v_\mu$ is bounded in $L^\infty(K_r)$ by Corollary \ref{corollaryBoundVelocity} and for a small time interval around zero $\nu(h) \in \Prob_2(K_r)$. Consequently, as $h$ goes to $0$,
$$
\begin{aligned}
\E_{\tau,r}(\nu(h)) - \E_{\tau,r}(\mu) & =  \int_{\Theta^2} v_{\mu}(\theta) \cdot (\tilde{\theta} - \theta) \, d (i \times \exp(-h \bold{P} v_{\mu}/ \| \bold{P} v_\mu \|_{L^2(\Theta; d\mu)}))\# \mu(\theta) \\
&+ o\left(h\right) \\
&= \int_{\Theta} v_{\mu}(\theta) \cdot \left(\exp\left(-h \bold{P} v_{\mu}(\theta)/ \| \bold{P} v_\mu \|_{L^2(\Theta; d\mu)}\right) - \theta\right) \, d \mu(\theta) + o(h).\\
\end{aligned}
$$
Dividing by $h$ and passing to the limit as $h$ goes to zero (justifying the passage to the limit as above), it holds that:

\begin{equation}\label{nuDiff1}
\lim_{h \rightarrow 0} \frac{\E_{\tau,r}(\nu(h)) - \E_{\tau,r}(\mu)}{h} = -\| \bold{P} v^\tau_\mu(\theta) \|_{L^2(\Theta; d\mu)}.
\end{equation}
Additionally, with~\eqref{nuLipschitz} :

\begin{equation}\label{nuDiff2}
\limsup_{h \rightarrow 0} \frac{\E_{\tau,r}(\nu(h)) - \E_{\tau,r}(\mu)}{W_2(\nu(h), \nu(0))} \leq -\| \bold{P} v_\mu \|_{L^2(\Theta; d\mu)}.
\end{equation}

To conclude :

\begin{itemize}

\item With~\eqref{nuDiff2} and~\eqref{boundNablaEP1}, the claim is proved :

$$
|\nabla_- \E_{\tau,r}|(\mu) = \| \bold{P} v_{\mu}\|_{L^2(\Theta; d\mu)}.
$$

\item Owing to this,~\eqref{nuLipschitz} and~\eqref{nuDiff2} the curve $[0,1] \ni h \mapsto \nu(h)$ is differentiable at $h=0$ by~\cite[Proof of (ii) Lemma 5.4]{Ohta2009} and :

$$
\nu^\prime(0) =  \nabla_- \E_{\tau,r}(\mu)/ |\nabla_- \E_{\tau,r}|(\mu).
$$

\end{itemize}

This finishes the proof of the lemma.

\end{proof}

\subsubsection{Existence without support limitation}

Note that for the moment the definition domain of $\E_{\tau,r}$ is reduced to measures supported in $K_r$. Using a bootstrapping argument, we will prove that the existence theorem~\ref{theoremGlobalWellPosedness} can be extended to the energy $\E_{\tau, +\infty}$.

%
%
%
%

\begin{proof}[Proof of Theorem~\ref{theoremGlobalWellPosedness}]
Let :

\begin{itemize}
\item $r_0>0$ be such that ${\rm Supp}(\mu_0) \subset K_{r_0}$,
\item $\mu^r: \mathbb{R}_+ \ni t \mapsto \mu^r(t)$ the gradient curve associated to $\E_{\tau, r}$ for $r>r_0$.
\end{itemize}
By Corollary~\ref{corollaryBoundVelocity}, it holds that $|v_{\mu^r(t)}(\theta)| \leq C r |\theta|$ for all $t\geq 0$. Hence, for all $\theta \in K_{r_0}$, $|\chi^r(t;\theta)| \leq r_0 e^{Cr t}$  for all time $t\in \left[0, T_r := \frac{1}{Cr} \log\left(\frac{r + r_0}{2r_0}\right)\right]$ and ${\rm Supp}(\mu^r(t)) \subset K_{(r + r_0)/2} \subset K_r$. By the definition of the gradient curve :

\begin{equation}\label{murGradientCurve}
\forall t \in [0, T_r], \ (\mu^r)^\prime(t) = \nabla_- \E_{\tau, r}(\mu^r(t)) = g^\prime(0)
\end{equation}
with $g:[0,1]\ni h \mapsto \exp\left(- \bold{P} v_{\mu^r(t)} h\right)$, by Lemma~\ref{lemmaIdentificationGradient}. Note that the right hand side of last equation does not depend explicitly on $r$ but on $\mu^r_\cdot$.

We construct the curve $\mu: [0,T_r] \to \Prob_2(\Theta)$ as follows:

$$
\forall t \in [0,T_r], \;  r >r_0 \ \mu(t) := \mu^r(t).
$$
This is well-defined since by uniqueness of the gradient curve with respect to $\E_{\tau, r}$, $\mu^{r_1}(t) = \mu^{r_2}(t)$ on $[0,\min(T_{r_1},T_{r_2})]$ for $r_0 < r_1 \leq r_2$. Defining for all $n\in \mathbb{N}^*$
$$
r_{n} := (n+1) r_0,
$$
we can build inductively a gradient curve on $\left[0, \frac{1}{C r_0}\sum_{i=1}^n \frac{1}{(i+1)} \log \left( \frac{i+2}{2(i+1)} \right) \right]$. As the width of this interval is diverging, it is possible to construct a gradient curve on $\R^+$.

All the properties given by the theorem comes from the properties of $\mu^r$ derived in Theorem~\ref{theoremWellPosednessCondSupport} and Proposition~\ref{propositionCharacteristics}.
%
%
%
%
%
%
\end{proof}

\begin{remark}\label{remarkNonexistenceUniqueness}
We make here two important remarks:

\begin{itemize}
\item We did not prove the existence of a gradient curve with respect to $\E_{\tau,\infty}$ because this functional is not proved to be convex along geodesics and it is impossible to define gradients without such an assumption.

\item The uniqueness of a solution to~\eqref{wellPosednessE} is out of the scope of this article. To prove it, one should link~\eqref{wellPosednessE} and the support condition to prove that locally in time, a solution to~\eqref{wellPosednessE} coincides with the unique gradient curve of $\E_{\tau,r}$ for some $r>0$ large enough. 
\end{itemize}
\end{remark}

\subsection{Link with backpropagation in neural network}\label{sectionLinkNN}

Here, we give a proof of Theorem~\ref{theoremLinkNN}.

\begin{proof}[Proof of Theorem~\ref{theoremLinkNN}]
Returning back to the proof of Theorem~\ref{theoremGlobalWellPosedness} and for all time $T>0$, one can find $r>0$ large enough such that $\mu$, $\mu_{m}$ coincide with gradient curves on $[0,T]$ with respect to $\E_{\tau, r}$ starting from $\mu_0$ and $\mu_{0,m}$ respectively. As gradient curves with respect to $\E_{\tau, r}$ verifies the following semigroup property~\cite[Theorem 5.11]{Ohta2009}

$$
\forall t \in [0,T], \ W_2(\mu(t), \mu_{m}(t)) \leq e^{\lambda_{\tau,r} t } W_2(\mu_0, \mu_{0,m}),
$$
the expected convergence on $C([0,T], \Prob_2(\Omega))$ holds by the convergence of initial measures.
\end{proof}

\subsection{Convergence of the measure towards the optimum}
\label{sectionConvergence}
In the following, a LaSalle's principle argument is invoked in order to prove Theorem~\ref{theoremConvergence}. For simplicity, we note  $\E_{\tau} := \E_{\tau, \infty}$ for $0<\tau<+\infty$.

\subsubsection{Characterization of optima}

In this part, we focus on a characterization of global optima.
For convenience, we extend the functional $\E_\tau$ to the set of signed finite measures on $\Theta$, denoted by $\mathcal M(\Theta)$.
\begin{lemma}
\label{lemmaMeasureMeasureProba}
For all $\mu \in \mathcal{M}(\Theta)$, there exists a probability measure $\mu_p$ such that $\E_\tau(\mu) = \E_\tau(\mu_p)$.
\end{lemma}

\begin{proof}
Let us first consider a positive signed measure $\mu \in \mathcal{M}^+(\Theta)$. If $\mu(\Theta) = 0$, $\Phi(\theta,.) = 0$ $\mu$-almost everywhere and $\E_\tau(\mu) = 0$. Taking $\mu_p := \delta_{(0,0,w,b)}$ with $w,b$ taken arbitrary \v{eis sufficient} to prove the desired result. Now, if $\mu(\Theta) \neq 0$, consider $\mu_p := T {\#} \left(\frac{\mu}{\mu(\Theta)}\right)$ where
$T: (c,a,w,b) \rightarrow (c\mu(\Theta),a\mu(\Theta),w, b)$. In this case :

$$
\begin{aligned}
\int_{\Theta} \Phi(\theta; \cdot) d\mu &= \int_{\Theta} \mu(\Theta) \Phi(\theta; \cdot) \frac{d \mu(\theta)}{\mu(\Theta)}\\
&=\int_{\Theta} \Phi(T \theta; \cdot) \frac{d \mu(\theta)}{\mu(\Theta)}\\
&= \int_{\Theta} \Phi(\theta; \cdot) d \mu_p(\theta)
\end{aligned}
$$
where we have used the form of $\Phi$~\eqref{definitionPhi}-\eqref{definitionPhi1} to get the last inequality.

Now take an arbitrary signed measure $\mu \in \mathcal{M}(\Theta)$. By Hahn-Jordan decomposition theorem, there exists $P,N$ $\mu$-measurable sets such that $P \cup N = \Theta$ and $\mu$ is non-negative (respectively non-positive) on $P$ (respectively $N$). The signed measure $\mu$ can be written as :
$$
\mu = \mu_P - \mu_N
$$
where $\mu_P,\mu_N \in \mathcal{M}^+(\Theta)$. Consider following map :

$$
G(c,a,w,b) :=
\left\{
\begin{array}{rl}
(-c,-a,w,b) & \text{if } (a,b,w,c) \in N\\
(c,a,w,b) & \text{if } (a,b,w,c) \in P\\
\end{array} 
\right.
$$
and the measure :

$$
\mu_G := G\# (\mu_P + \mu_N) \in \mathcal{M}^+(\Theta).
$$
By construction, we have $P_\tau\left(T\# \left( \frac{\mu_G}{\mu_G(\Theta)} \right) \right) = P_\tau(\mu)$ and consequently, $\E_\tau(\mu) = \E_\tau\left(T\# \left( \frac{\mu_G}{\mu_G(\Theta)} \right) \right)$.
\end{proof}

\begin{lemma}
\label{lemmaOptimality}
The measure $\mu \in \Prob_2(\Theta)$ is optimal for Problem~\ref{poissonProba} if and only if $\phi_\mu(\theta) = 0$ for all $\theta \in \Theta$.
\end{lemma}

\begin{proof}
Suppose $\mu \in \Prob_2(\Theta)$ optimal and let $\zeta \in L^1(\Theta; \mu)$. Then, for all $\nu := \zeta \mu + \nu^\perp \in \mathcal{M}(\Theta)$  (Lebesgue decomposition of $\nu$ with respect to $\mu$ with $\zeta \in L^1(\Theta; \mu)$) and owing to Lemma~\ref{lemmaMeasureMeasureProba}, as $t$ goes to $0$,

$$
\begin{aligned}
\E_\tau(\mu + t\nu) &= \E(P_\tau(\mu) + t P_\tau(\nu))\\
&= \E_\tau(\mu) + td \E|_{P_\tau(\mu)}(P_\tau(\nu)) + o(t).
\end{aligned}
$$
Hence as $\mu$ is optimal

$$
\begin{aligned}
0 = \frac{d}{dt}\left[\E_\tau(\mu + t\nu)\right]|_{t=0} &= d \E|_{P_\tau(\mu)}(P_\tau(\nu)) \\
&= \int_\Theta \ d \E|_{P_\tau(\mu)}(\Phi_\tau(\theta; \cdot)) d \nu(\theta) \\
&= \int_\Theta \phi_\mu(\theta) d \nu(\theta)\\
&= \int_\Theta \phi_\mu(\theta) \zeta(\theta) d \mu(\theta) + \int_\Theta \phi_\mu(\theta) d \nu^\perp(\theta).
\end{aligned}
$$
As this is true for all $\zeta \in L^1(\Theta, \mu)$, one gets:

\begin{equation}
\label{phimupositive}
\phi_\mu = 0 \ \mu\text{-almost everywhere}, \quad \phi_\mu = 0 \ \nu^\perp\text{-almost everywhere}
\end{equation}
for all $\nu^\perp \perp \mu$. As $\phi_\mu$ is continuous, this is equivalent to $\phi_\mu = 0$ everywhere in $\Theta$. Indeed, let $\theta \in \Theta$. If $\theta$ belongs to ${\rm Supp}(\mu)$, then by definition of the support, $\mu(B(\theta,\varepsilon))>0$  for all $\varepsilon>0$. Thus, one can take $\theta_\varepsilon \in B(\theta,\varepsilon)$ with $\phi_\mu(\theta_\varepsilon) = 0$. As $\displaystyle \theta_\varepsilon \mathop{\longrightarrow}_{\varepsilon \to 0} \theta $, using the continuity of $\phi_\mu$, we obtain $\phi_{\mu}(\theta) = 0$. If $\theta \not\in {\rm Supp}(\mu)$, then $\delta_{\theta} \perp \mu$ and necessarily, $\phi_\mu(\theta) = 0$. The reverse implication is trivial. 

Conversely suppose now $\phi_\mu = 0$ everywhere in $\Theta$ and take $\nu \in \Prob_2(\Theta)$, then by previous computations and the convexity of $\E$ (slopes are increasing)

$$
0 = \frac{d}{dt}\left[ \E(\mu + t(\mu-\nu))\right] = \frac{d}{dt}\left[ \E(P_\tau(\mu) + t P_\tau(\mu - \nu))\right]\leq \E(P_\tau(\nu)) - \E(P_\tau(\mu)) 
$$
which implies that 
$$
\E_\tau(\mu) \leq \E_\tau(\nu)
$$
and $\mu$ is optimal.

\end{proof}

\subsubsection{Escape from critical points}

In this section, we use the notation :

$$
\theta = (a,c,w,b) =: (a,c, \omega)
$$
to make the difference between "linear" variables and "nonlinear" ones.

%
%

\begin{lemma}
\label{lemmaContinuityPhi}
For all $\mu,\nu$ in $\Prob_2(\Theta)$, it holds that
$$
\forall \theta  \in \Theta, \ |\phi_\mu(\theta) - \phi_\nu(\theta)| \leq C\left(\int_{\Theta} |\theta_1|^2 d\mu(\theta_1) + \int_{\Theta} |\theta_2|^2 d\nu(\theta_2) \right) W^2_2(\mu,\nu)(1+ |\theta|^2)
$$

$$
\forall \theta  \in \Theta, \ |v_\mu(\theta) - v_\nu(\theta)| \leq C\left(\int_{\Theta} |\theta_1|^2 d\mu(\theta_1) + \int_{\Theta} |\theta_2|^2 d\nu(\theta_2) \right) W^2_2(\mu,\nu) (1+|\theta|^2)
$$
\end{lemma}

\begin{proof}
Here we focus on $v_\mu$, the proof for $\phi_\mu$ being very similar. 
Considering~\eqref{definitionV1}-\eqref{definitionV2}, one can decompose $v_\mu$ as
\begin{equation}\label{eq:vmu}
v_\mu =: v_{\mu,1} + v_{2} + v_{\mu,3}, 
\end{equation}

with 
\begin{align*}
v_{\mu,1} &:= \nabla_\theta\left[\langle \nabla_x P_{\tau}(\mu), \nabla_x \Phi_{\tau}(\theta;\cdot) \rangle_{L^2(\Omega)}\right],\\
v_2 & :=\nabla_\theta \left[ - \langle f,  \Phi_{\tau}(\theta;\cdot)   \rangle_{L^2(\Omega)} \right] ,\\
v_{\mu,3} &:= \nabla_\theta \left[ \int_{\Omega}  P_{\tau}(\mu)(x) dx \times \int_{\Omega} \Phi_{\tau}(\theta;x) dx\right] .\\
\end{align*}
Using standard integral derivation and Fubini theorems, it holds that for all $\gamma \in \Gamma_o(\mu,\nu)$,
$$
\begin{aligned}
v_{\mu,1}(\theta) - v_{\nu,1}(\theta) &= \int_{\Theta^2} \int_{\Omega} \nabla_\theta \nabla_x \Phi_\tau(\theta; x) ( \nabla_x \Phi_\tau(\theta_1; x) -  \nabla_x \Phi_\tau(\theta_2; x))  dx d\gamma(\theta_1,\theta_2).
\end{aligned}
$$
Owing to~\eqref{gradx}-\eqref{gradthetax}, one gets 
$$
\begin{aligned}
|v_{\mu,1}(\theta) - v_{\nu,1}(\theta)| & \leq C(\tau) \int_{\Theta^2} \max(|\theta_1|,|\theta_2|) |\theta_1 - \theta_1| |\theta|^2 dx d\gamma(\theta_1,\theta_2)\\
&\leq C(\tau)\left(\int_{\Theta} |\theta_1|^2 d\mu + \int_{\Theta} |\theta_2|^2 d\nu \right) W^2_2(\mu,\nu) |\theta|^2,
\end{aligned}
$$
where $C(\tau)$ is a positive constant which only depends on $\tau$, and where we used the Cauchy-Schwartz inequality. For the third term in the decomposition (\ref{eq:vmu}), one has :

$$
v_{\mu,3} - v_{\nu,3} = \int_{\Theta^2}\int_{\Omega}  \Phi_{\tau}(\theta_1;\cdot) - \Phi_{\tau}(\theta_2;\cdot) dx d\gamma(\theta_1, \theta_2) \times \int_{\Omega} \nabla_\theta\Phi_{\tau}(\theta;\cdot) dx.
$$
Owing to~\eqref{gradtheta}, one gets :

$$
\begin{aligned}
|v_{\mu,3}(\theta) - v_{\nu,3}(\theta)| &\leq C(\tau) \int_{\Theta^2} \int_{\Omega} \max(|\theta_1,\theta_2|) |\theta_1 - \theta_1|  dx d\gamma(\theta_1,\theta_2) |\theta|\\
&\leq  C(\tau)\left(\int_{\Theta} |\theta_1|^2 d\mu + \int_{\Theta} |\theta_2|^2 d\nu \right) W^2_2(\mu,\nu) |\theta|
\end{aligned}
$$
where we used again the Cauchy-Schwartz inequality. Hence the desired result. 
\end{proof}

\begin{proposition}
\label{propositionEscape}
Let $\mu \in \Prob_2(\Theta)$ such that there exists $\theta\in\Theta$, $\phi_\mu(\theta) \neq 0$.
Then there exist a set $A \subset \Theta$ and $\varepsilon>0$ such that if there exists $t_0>0$ with $W_2(\mu({t_0}),\mu)\leq\varepsilon$ and $\mu({t_0})(A)>0$, then there exists a time $0 < t_0 < t_1 < +\infty$ such that $W_2(\mu({t_1}),\mu)>\varepsilon$.

\end{proposition}

\begin{proof}
As $\phi_\mu$ is linear in $a$ and $c$, it can be written under the form
$$
\phi_\mu(\theta) =: a \psi_{\mu}(\omega) +  cr_{\mu}.
$$

By hypothesis, the set 
$$
A_0 := \{ \theta \in \Theta \ | \ \phi_\mu(\theta) \neq0 \}
$$ 
is a non empty (open set). This is equivalent to say that either there exists $\omega$ such that $\psi_{\mu}(\omega)  \neq 0$ or $r_{\mu} \neq 0$. Suppose that $\psi_{\mu}\neq 0$ is non zero somewhere, the case for $r_{\mu}$ being similar. For all $\alpha \in \mathbb{R}$, we denote by
$$
\left\{
\begin{aligned}
A_\alpha^+ &= \psi_{\mu}^{-1}(]0, +\infty[),\\
A_\alpha^- &= \psi_{\mu}^{-1}(]-\infty, 0[).
\end{aligned}
\right.
$$
Now we focus on $A_0^-$ and suppose that this set is non empty. The case where $A_0^+$ is non empty is similar to handle and left to the reader. 

By Lemma~\ref{lemmaRegularValuesPhi} and the regular value theorem, there exists $\eta >0$ such that $\partial A^-_{-\eta} = \psi_\mu^{-1}(\{-\eta\})$ is a $(d+1)-$orientable manifold on which $\nabla_{\omega} \psi_\mu$ is non zero. With our choice of activation function $\sigma_{H,\tau}$, it is easy to prove that $A^-_{-\eta}$ is a bounded set. Indeed, if $b$ is large enough, then $\Omega \ni x \mapsto \sigma_{H,\tau}(w\cdot x +b)$ is zero and $\psi_\mu(w,b)$ is zero.

On $A^-_{-\eta}$, the gradient $\nabla_\omega \psi_\mu$ is pointing outward $A^-_{-\eta}$ and, denoting by $n_{\rm out}$ the outward unit vector to $A^-_{-\eta}$, there exists $\beta >0$ such that $|\nabla_{\omega} \psi_\mu \cdot n_{\rm out} | > \beta$ for  on $ \partial A^-_{-\eta}$, since this continuous function is nonzero on a compact set. Hence, defining : 
$$
A := \{ (a,c,\omega)\in \Theta \ | \ \omega \in A^-_{-\eta}, \ a × \geq 0 \}
$$
and owing to the fact that $v_{\mu} = (v_{\mu,a}, v_{\mu,c}, v_{\mu, \omega})$ with $v_{\mu,a} = \psi_\mu(\omega)$, $v_{\mu, c} = r_\mu$, $v_{\mu,\omega} = a \nabla_\omega \psi_\mu(\omega)$, it holds :

\begin{equation}\label{boundVmu}
\left\{
\begin{aligned}
v_{\mu,a } & <  \eta \text{ on }A\\
v_{\mu,\omega }\cdot n_{out} & >  \beta a \text{ on }   \mathbb{R}_+ \times \mathbb{R} \times \partial A^-_{-\eta}.
\end{aligned}
\right.
\end{equation}

By contradiction, suppose that $\mu(t_0)$ has non zero mass on $A$ and that $W_2(\mu,\mu(t)) \leq \varepsilon$ (with $\varepsilon$ fixed later) for all time $t\geq t_0$.  Then using Lemma~\ref{lemmaContinuityPhi}, one has :

\begin{equation}\label{comparisonVmuVmut}
|v_{\mu(t)}(\theta) - v_\mu(\theta)| \leq C(\tau,\mu) (1+|\theta|^2) \varepsilon
\end{equation}
and
$$
|\phi_{\mu(t)}(\theta) - \phi_\mu(\theta)| \leq C(\tau,\mu)(1+|\theta|^2) \varepsilon.
$$
One takes $\varepsilon := \sqrt{\frac{\eta}{2 C(\tau,\mu) R}}$ where $R>0$  satisfies :
\begin{equation}\label{eq:condr}
(R-1) \mu({t_0})(A) > \int |\theta|^2 d \mu + \frac{\eta}{2 C(\tau,\mu) R}
\end{equation}
which exists since $\mu(t_0)(A) >0$ by hypothesis. On the set $\{ \theta \in A \ | \ 1 + |\theta|^2 \leq R \}$ and by~\eqref{comparisonVmuVmut}, we have :

$$
|v_{\mu(t)}(\theta) - v_\mu(\theta)| \leq \frac{\eta}{2}
$$
and so by~\eqref{boundVmu} and the fact that $v_t = - v_{\mu(t)}$:

$$
\left\{
\begin{aligned}
v_{t,a } & >  \eta/2 \text{ on }A\\
v_{t,\omega }\cdot n_{out} & <  -\beta/2 \times a \text{ on }  \partial A^-_{-\eta}.
\end{aligned}
\right.
$$

The general picture is given by Figure~\ref{figureDescriptionEscapeMass}.
As a consequence, there exists a time $t_1$ such that the set $\{ \theta \in A \ | \ 1 + |\theta|^2 \leq R \}$ has no mass and 

$$
\int |\theta|^2 d\mu(t)(\theta) \geq (R -1) \mu(t)(A) \geq (R -1) \mu(t_0)(A).
$$
At the same time, as $W_2(\mu,\mu(t)) \leq \varepsilon$ :

$$
\int |\theta|^2 d\mu(t)(\theta) \leq \int |\theta|^2 d\mu(\theta) + \varepsilon^2 = \int |\theta|^2 d\mu(\theta) + \frac{\eta}{2 C(\tau,\mu)}
$$
and this a contradiction with the condition (\ref{eq:condr}) on $R$.

\end{proof}

\begin{remark}
The set $A$ constructed in the proof of previous lemma is of the form :
\begin{equation}
\label{setSeparability}
A := \{ (a,c, \omega) \in \Theta \ | \ \omega \in A^-_{-\eta_1}\} \cup \{ (a,c,\omega) \ | \ \omega \in A^+_{\eta_2} \}
\end{equation}
where $\eta_1, \eta_2$ are strictly positive.
\end{remark}

\begin{lemma}
\label{lemmaRegularValuesPhi}
For all $\mu \in \Prob_2(\Theta)$, if $\psi_\mu < 0$ somewhere, there exists a strictly negative regular value $-\eta$ ($\eta>0$) of $\psi_\mu$.
\end{lemma}

\begin{proof}

As $\psi_\mu < 0$ somewhere and by continuity, there exists a non empty open $O \subset ]-\infty,0[$ such that $O \subset range(\psi_\mu)$. Next, we use the Sard-Morse theorem recalled below :

\begin{theorem}[Sard-Morse]
Let $\mathcal{M}$ be a differentiable manifold and $f: \mathcal{M} \rightarrow \R$ of class $C^n$, then the image of the critical points of $f$ (where the gradient is zero) is Lebesgue negligible in $\R$.
\end{theorem}
This result applies to $\phi_\mu$ and the image of critical points of $\phi_\mu$ is Lebesgue negligible. As a consequence, there exists a point $o \in O$ which is a regular value of $\phi_\mu$. As $o \in O$, it is strictly negative and this finishes the proof of the lemma.

\end{proof}

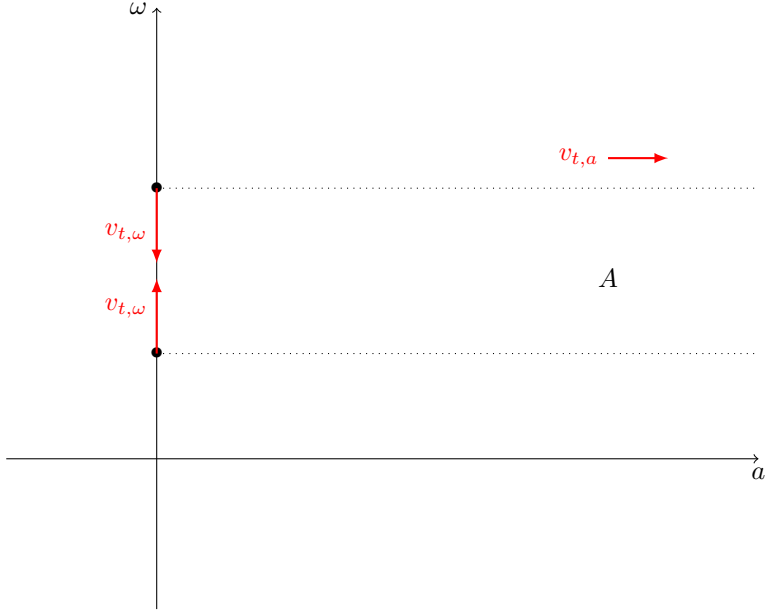
\begin{figure}
\centering
\begin{tikzpicture}[scale=2]
\draw[->] (-1,0) -- (4,0);
\draw (4,0) node[below] {$a$};
\draw [->] (0,-1) -- (0,3);
\draw (0,3) node[left] {$\omega$};

\draw [-, dotted] (0,0.5 + 1/5) -- (4,0.5 + 1/5);
\draw [-, dotted] (0,1 + 4/5) -- (4,1 + 4/5);

\draw (0,0.5 + 1/5) node[] {$\bullet$};
\draw (0,1 + 4/5) node[] {$\bullet$};
\draw (3,1.2) node[] {$A$};

\draw [->,red, arrows={-latex}, thick] (0,0.5 + 1/5) -- (0,1.2);
\draw (0,1) node[left] {$\comment{v_{t,\omega}}$};
\draw [->,red, arrows={-latex},thick] (0,1 + 4/5) -- (0,1.3);
\draw (0,1.5) node[left] {$\comment{v_{t,\omega}}$};
\draw [->,red, arrows={-latex}, thick] (3,2) -- (3.4,2);
\draw (3,2) node[left] {$\comment{v_{t,a}}$};

\end{tikzpicture}
\caption{The escape of mass towards large values of $a$}
\label{figureDescriptionEscapeMass}
\end{figure}

\subsubsection{Convergence}
%

This preliminary lemma gives an insight of why Hypothesis~\ref{hypothesisSupport} is useful :
\begin{lemma}\label{preliminaryLemmaPhi}
For all $\mu \in \Prob_2(\Theta)$, $\theta \notin \R^2 \times S_{\R^d}(1) \times ]-\sqrt{d} -2,\sqrt{d} + 2[, \tau >1$, the potential writes :

$$
\phi_\mu(\theta) = c r_\mu 
$$
where $r_\mu$ is a constant that depends on $\mu$. In particular, $\phi_\mu(\theta)$ does not depend on $a,w,b$.
\end{lemma}

\begin{proof}
For all $x \in \Omega, |b| > \sqrt{d} + 2, \tau >1$ :

$$
|w \cdot x + b| \geq |b| - |x|_\infty |w|_1 > 2
$$
and
$$
\sigma_{H,\tau} (w \cdot x + b) = 0.
$$ 
This implies that for $|b| \geq \sqrt{d} + 2, \mu \in \Prob_2(\Theta)$, the potential $\phi_{\mu}$ writes $\phi_{\mu}= c r_{\mu}$ where $r_\mu$ is a constant. 
\end{proof}

In fact Hypothesis~\ref{hypothesisSupport} is verified by the gradient curve $(\mu(t))_{t\geq 0}$ for all time. This is proved in the next lemma. 

\begin{lemma}\label{lemmaPositiveMeasure}
If $\mu_0$ satisfies Hypothesis~\ref{hypothesisSupport} then for all $t\geq 0$ and all open set $O \subset S_{\R^d}(1) \times [-\sqrt{d} -2,\sqrt{d} + 2]$, 

$$
\mu(t)(\R^2 \times O) > 0
$$

\end{lemma}

The arguments of the proof of last lemma are based on fine tools of algebraic topology. One can find a nice introduction to the topic in the reference book~\cite{HatcherBook}. With simple words, we enjoy the homotopy properties on the sphere to prove that the measure $\mu(t)$ keeps a large enough support.

\begin{proof}
For all $t\geq 0$, as $\mu(t) = (\chi(t))\# \mu_0$, we have~\cite[Lemma C.8]{BachChizat2018} :

\begin{equation}\label{supportByFlow}
{\rm Supp}(\mu(t)) = \overline{\chi(t) \left({\rm Supp} (\mu_0)\right)}.
\end{equation}
Now let $\xi_t(w,b) := (P_{S_{\R^d}(1) \times \R } \circ \chi(t))((0,0,w,b))$ where $P_{S_{\R^d}(1) \times \R }$ is the projection on $S_{\R^d}(1) \times \R$ ($w,b$ variables). We claim that the choice of the function of activation lets the extremal spheres invariant \textbf{ie} $\xi_t(w,\pm(\sqrt{d} + 2)) = (w,\pm(\sqrt{d} + 2))$. Indeed, by Lemma~\ref{preliminaryLemmaPhi} for $\theta = (c,a,w,\pm(\sqrt{d} + 2))$, $\phi_\mu(\theta) = c r_\mu$ giving :

$$
\left\{
\begin{array}{rl}
v_{\mu,w}(\theta) =& 0, \\
v_{\mu,b}(\theta) =& 0
\end{array} 
\right.
$$
and the claim is proven. Consequently by Lemma~\ref{lemmaSurjectivity}, the continuous map $\xi_t$ is surjective.

Now let $O \subset S_{\R^d}(1) \times [-\sqrt{d} - 2,\sqrt{d} + 2]$ be an open set. By what precedes, there exists a point $\omega \in S_{\R^d}(1) \times [-\sqrt{d} - 2,\sqrt{d} + 2]$ such that $\xi_t (\omega) \in O$ and $\chi(t)((0,0, \omega)) \in \R^2 \times O$. As $(0,0, \omega)$ belongs to the support of $\mu_0$ by hypothesis then $\chi(t)((0,0, \omega))$ belongs to the support of $\mu(t)$ by~\eqref{supportByFlow} and :

$$
\mu(t)(\R^2 \times O) > 0
$$
which finishes the proof of the lemma.

\end{proof}

Lemma~\ref{lemmaSurjectivity} gives conditions for the surjectivity of a continuous map on a cylinder.
\begin{lemma}\label{lemmaSurjectivity}
Let $f$ be a continuous map $f: S_{\R^d}(1) \times [0,1] \rightarrow S_{\R^d}(1) \times [0,1] =:C$, homotopic to the identity such that :

$$
\forall w \in S_{\R^d}(1), \ 
\left\{
\begin{aligned}
f(w,0) =& (w,0), \\
f(w,1) =& (w,1).
\end{aligned}
\right.
$$
Then $f$ is surjective.

\end{lemma}
\begin{proof}
Suppose that $f$ misses a point $p$, then necessarily $p = (w,t)$ with $0 < t < 1$. We can write :

$$
g: C \rightarrow C \setminus \{p\}
$$
the restriction of $f$ on its image. The induced homomorphism on homology groups writes :

$$
g_\star: H_{d-1}(C) \rightarrow H_{d-1}(C \setminus \{p\}).
$$

Aside that, we have the classic information on homology groups of $C$ and $C \setminus \{p\}$ :

$$
\left\{
\begin{aligned}
H_{d-1}(C)&= H_{d-1}(S_{\R^d}(1)) &\simeq \mathbb{Z}, \\
H_{d-1}(C \setminus \{p\}) &= H_{d-1}(S_{\R^d}(1) \vee S_{\R^d}(1)) &\simeq \mathbb{Z}^2
\end{aligned}
\right.
$$
where $\vee$ designates the wedge sum. Thus, the homomorphism $g_\star$ can be written as : 
$$
g_\star: \mathbb{Z} \rightarrow \mathbb{Z}^2.
$$

As $g$ lets the two spheres $w \rightarrow (w,0), w \rightarrow (w,1)$ invariant, we have :

$$
g_\star(1) = (1,1).
$$
Now we note $i: C \setminus \{p\} \rightarrow C$ the canonical inclusion map. For all $(a,b)\in \mathbb{Z}^2$, 

$$
i_\star(a,b) = a + b.
$$

By hypothesis, $f$ is homotopic to the identity so $f_\star = I_\star$ and 
$f_\star(1) = 1$ but at the same time :

$$
f_\star(1) = i_\star g_\star (1) = i_\star((1,1)) = 2
$$
which gives a contradiction.

\end{proof}

It allows to conclude on the convergence and prove Theorem~\ref{theoremConvergence}.


\begin{proof}[Proof of Theorem~\ref{theoremConvergence}]
By contradiction, suppose $\mu^\star$ is not optimal. Then by Lemma~\ref{lemmaOptimality}, $\phi_{\mu^\star} \neq 0$ somewhere. Reusing the separation of variables (see the proof of Proposition~\ref{propositionEscape}), $\phi_{\mu^\star}$ writes :

$$
\phi_{\mu^\star}(\theta) = a \psi_{\mu}(w,b) +  cr_{\mu}.
$$

Hence either :

\begin{itemize}
\item $r_\mu$ is not zero and $v_{\mu,c} \neq 0$ and one can prove that some mass escapes at $c = \infty$ as in the proof of Proposition~\ref{propositionEscape}.

\item $\psi_{\mu}$ is not identically zero and the set $A$ defined in~\eqref{setSeparability} is not empty and verifies :

\begin{equation}\label{AincludedHypo}
A \subset \R^2 \times S_{\R^d(1)} \times [-\sqrt{d}-2, \sqrt{d}+2]
\end{equation}
by Lemma~\ref{preliminaryLemmaPhi}.

\end{itemize}

We focus on the last item. By Proposition~\ref{propositionEscape}, there exists $\varepsilon>0$ such that if $W_2(\mu_{t_0}, \mu^\star) \leq \varepsilon$ for some $t_0$ and $\mu(t_0)(A) > 0$ then there exists a further time $t_1$ with $W_2(\mu(t_0), \mu^\star) > \varepsilon$. As $(\mu(t))_{t\geq 0}$ converges towards $\mu^\star$, there exists $t_0$ such that :

$$
\forall t \geq t_0, \ W_2(\mu(t_0), \mu^\star) \leq \varepsilon.
$$
But by Lemma~\ref{lemmaPositiveMeasure} and~\eqref{AincludedHypo}, for all time $\mu(t)(A) > 0$ and consequently there exists a time $t_1>t_0$ with :

$$
\ W_2(\mu({t_0}), \mu^\star) > \varepsilon
$$
which gives the contradiction.

\end{proof}

\section{Numerical experiments}\label{sectionNumericalExperiment}
In this section, we will conduct numerical experiments to evaluate the potential of the proposed method.

\subsection{The effect of frequency}

First, the influence of the frequency on the approximation is investigated. To do so, we consider $d=1$ and the following source term for which the solution is a cosinus mode :

$$
f_k(x) := \pi^2 |k|^2 \cos(\pi k \cdot x).
$$
In higher dimension, we use the corresponding source term which is a tensor product of its one dimensional counterpart :

$$
f_k(x_1, \cdots, x_d) := \pi^2 |k|_{l^2}^2 \cos(\pi k_1 \cdot x_1) \cdots \cos(\pi k_d \cdot x_d).
$$

The \href{https://gitlab.com/mathias.dus31/ml_edp/-/releases/v1.0.1}{code} is written using python supplemented with Keras/Tensorflow framework. One should remember the following implementation facts :

\begin{itemize}
\item The neural network represents the numerical approximation taking values of $x \in \Omega$ as input and giving a real as output.

\item The loss function is approximated with a Monte Carlo sampling for the integrals where the measure is uniform on $\Omega$. For each training phase, we use batches of size $10^2$ obtained from a dataset of $10^5$ samples, the number of epochs is calculated to have a time of optimization equals to $2$ (learning rate $\times$ number steps $= 2$). Note that the dataset is shuffled at each epoch.

\item The derivative involved in the loss is computed thanks to automatic differentiation.

\item The training routine is given by the algorithm of backpropagation coupled with a gradient descent optimizer for which the learning rate $\zeta := \frac{1}{2 n m}$ where $n$ is the batch size and $m$ is the width of the neural network involved. This choice will be explained later in the analysis.

\item In all the plots, the reader will see the mean curve and a shaded zone representing the interval whose width is two times the standards deviation. Each simulation is run 4 times to calculate these statistical parameters.
\end{itemize}

For $d=1$ and a width $m = 1000$, the simulations are reported in Figure~\ref{figureDim1} for which very satisfactory results for $k=1,3$ are observed, the same conclusions hold for $d=2$. 


\begin{figure}[H]
    \begin{subfigure}{.5\textwidth}
    	\centering
        \includegraphics[width=.8\linewidth]{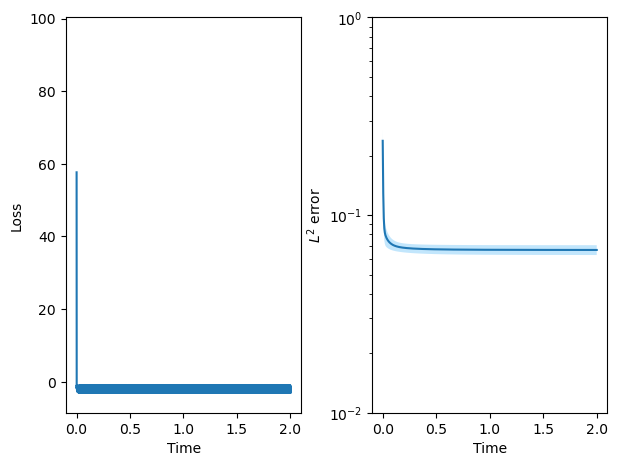}
        \caption{The case $d=1$ and $k=(1)$}
    \end{subfigure}
    \begin{subfigure}{.5\textwidth}
    	\centering
        \includegraphics[width=.8\linewidth]{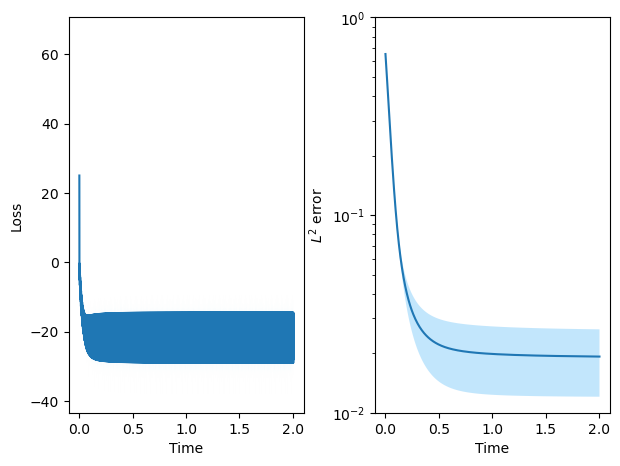}
        \caption{The case $d=1$ and $k=(3)$}
    \end{subfigure}
    \begin{center}
    \begin{subfigure}{.5\textwidth}
        \centering
        \includegraphics[width=.8\linewidth]{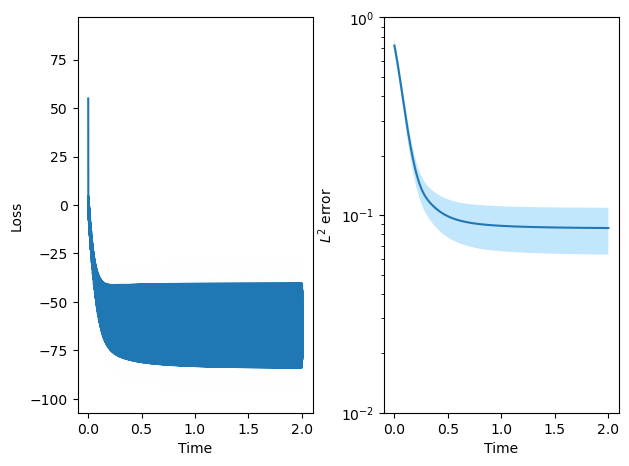}
        \caption{The case $d=1$ and $k=(5)$}
    \end{subfigure}
    \end{center}
    \caption{The effect of frequency on the approximation when $d=1$ and $m=1000$}
    \label{figureDim1}
\end{figure}

\begin{figure}[H]
    \begin{subfigure}{.5\textwidth}
    	\centering
        \includegraphics[width=.8\linewidth]{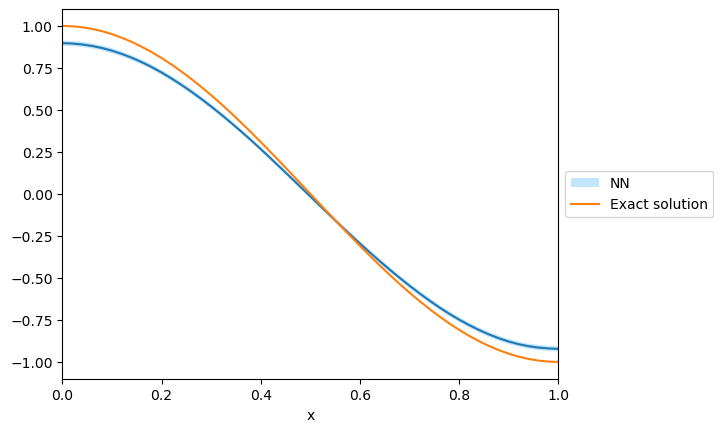}
        \caption{The case $d=1$ and $k=(1)$}
    \end{subfigure}
    \begin{subfigure}{.5\textwidth}
    	\centering
        \includegraphics[width=.8\linewidth]{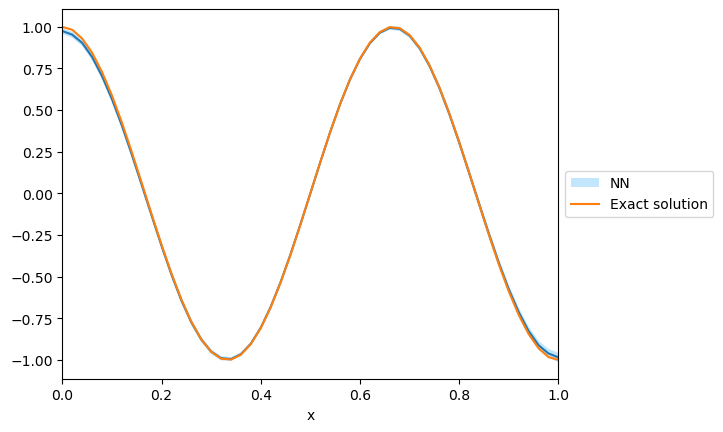}
        \caption{The case $d=1$ and $k=(3)$}
    \end{subfigure}
    \begin{center}
    \begin{subfigure}{.5\textwidth}
        \centering
        \includegraphics[width=.8\linewidth]{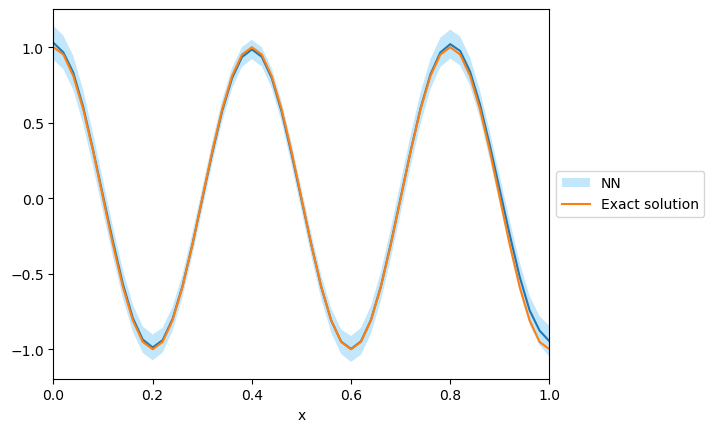}
        \caption{The case $d=1$ and $k=(5)$}
    \end{subfigure}
    \end{center}
    \caption{The numerical solutions when $d=1$ and $m=1000$}
    \label{figureDim1_sol}
\end{figure}

\begin{figure}[H]
    \begin{subfigure}{.5\textwidth}
    	\centering
        \includegraphics[width=.8\linewidth]{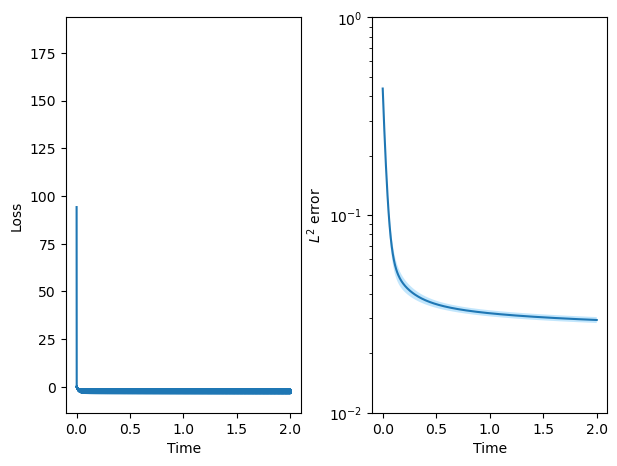}
        \caption{The case $d=2$ and $k=(1,1)$}
    \end{subfigure}
    \begin{subfigure}{.5\textwidth}
    	\centering
        \includegraphics[width=.8\linewidth]{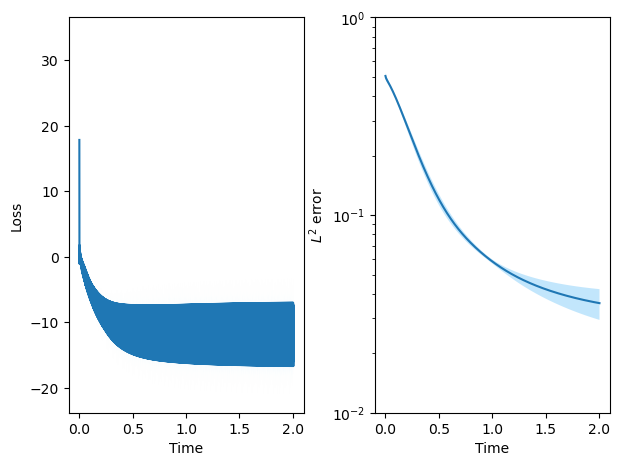}
        \caption{The case $d=2$ and $k=(3,1)$}
    \end{subfigure}
    \begin{center}
    \begin{subfigure}{.5\textwidth}
        \centering
        \includegraphics[width=.8\linewidth]{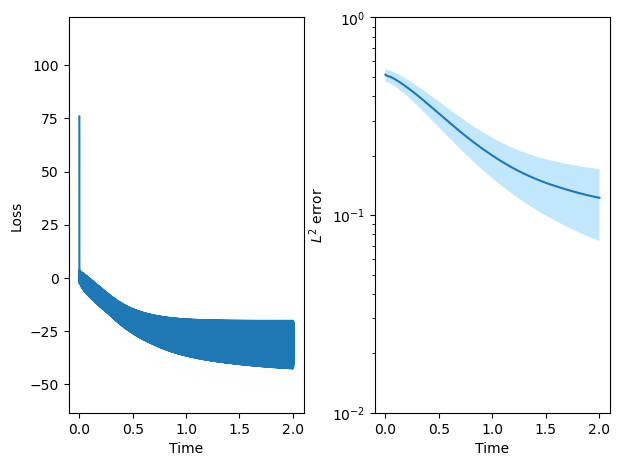}
        \caption{The case $d=2$ and $k=(5,1)$}
    \end{subfigure}
    \end{center}
    \caption{The effect of frequency on the approximation when $d=2$}
    \label{figureDim2}
\end{figure}

\begin{remark}
In this remark, we expose some heuristic arguments for the present choice of scaling related to the learning rate :

$$
\xi := \frac{1}{2 n m}.
$$
It is possible to write the learning scheme as follows :

\begin{equation}\label{batch_size_gradient_descent}
\frac{\theta_{t+1} - \theta_{t}}{dt} = -\nabla_\theta \phi^n_{\mu^m_t}(\theta_t)
\end{equation}
where :

\begin{equation}\label{def_pot_batch_size}
\phi^n_{\mu^m_t}(\theta) := \frac{1}{nm} \sum_{i,j} \nabla \Phi(\theta_j, x_i) \cdot \nabla \Phi(\theta, x_i) - f(x_i) \Phi(\theta, x_i)  + \left( \frac{1}{nm} \sum_{i,j} \Phi(\theta, x_i) \right)^2
\end{equation}
with $(x_i)_i$ are $n$ samples taken uniformly on the $d$ dimensional cube. 

By analogy, equations~\eqref{batch_size_gradient_descent}-\eqref{def_pot_batch_size} can be interpreted as an explicit finite element scheme for the heat equation where the space discretization parameter is $
h := \frac{1}{\sqrt{nm}}.
$
This gives the CFL condition :

$$
2dt \leq h^2
$$
which is equivalent to :

$$
dt \leq \frac{1}{2 n m }.
$$
In practice, one can observe that if one takes $dt > O\left(\frac{1}{n m}\right)$ then the scheme diverges in the same way as a classic finite elements scheme. 

The CFL condition is bad news since it prevents the use of large batch sizes necessary to get a good precision. In practice, the maximum on can do with a standard personal computer is $n,m=10^2$.

\end{remark}

%
%
%

\subsection{The effect of dimension}

To evaluate the effect of dimension on performance, we consider frequencies of the form $k = (\bar{k}, 0, \cdots, 0)$ where $\bar{k}$ is an integer, and plot the $L^2$ error as a function of the dimension for different $\bar{k}$. This is done in Figure~\ref{figureEffectDim} where several observations can be made :

\begin{itemize}

\item For low frequency, the precision is not affected by dimension.

\item At high frequency, performance are deteriorated as dimension increases.

\item Having a larger neural network captures better high frequency modes up to a certain dimension.

\item Variance increases with frequency but not with dimension.

\end{itemize}

\begin{figure}[H]
    \begin{subfigure}{.5\linewidth}
    	\centering

        \includegraphics[width=.8\linewidth]{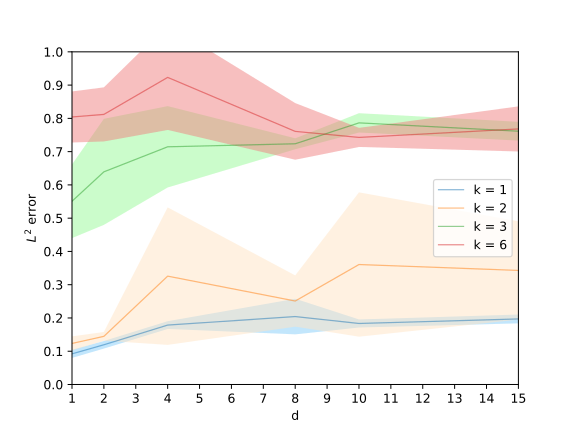}

        \caption{$m=10$}
    \end{subfigure}
    \begin{subfigure}{.5\linewidth}
    	\centering
        \includegraphics[width=.8\linewidth]{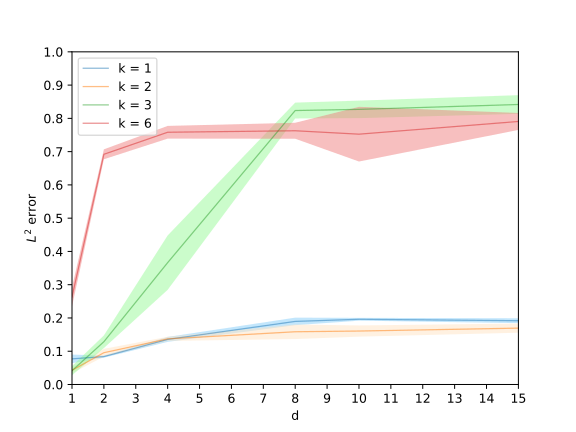}

        \caption{$m=100$}
    \end{subfigure}
    \begin{center}
    \begin{subfigure}{.5\linewidth}
        \centering
        \includegraphics[width=.8\linewidth]{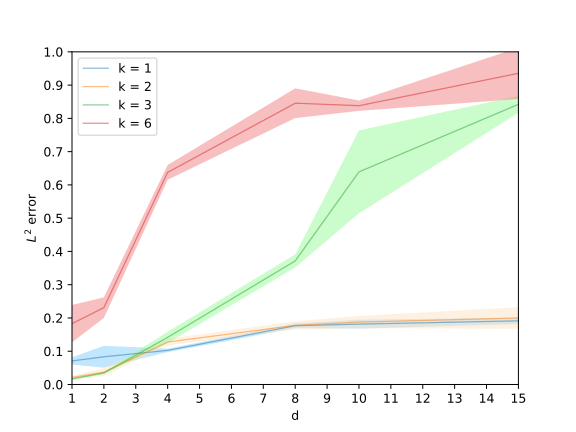}
        \caption{$m=1000$}
    \end{subfigure}
    \end{center}
    \caption{The effect of dimension for different frequencies and width}
    \label{figureEffectDim}
\end{figure}

For completeness we plot in Figure~\ref{highDimLowFrq} a high dimensional example where $d = 10$, $k=(1,1,0,\cdots,0)$ to show that the proposed method works well in the high dimensional/low frequency regime. The contour plot shows the function's values on the slice $(x_1,x_2, 0.5, \cdots, 0.5)$.

\begin{figure}[H]
\centering
\includegraphics[width=.5\linewidth]{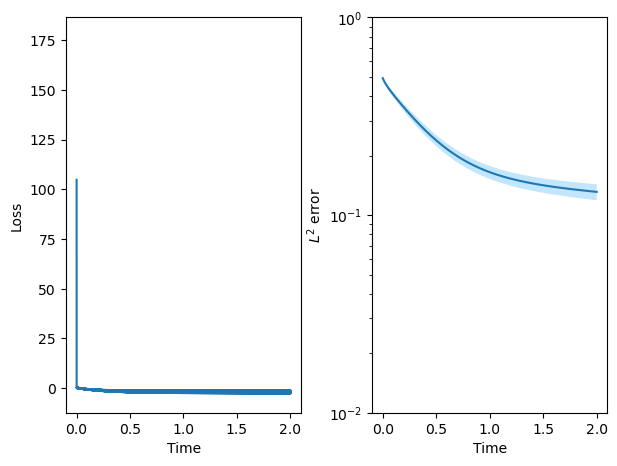}
\caption{The case $d=10$, $k=(1,1,0,\cdots,0)$ and $m=1000$}
\label{highDimLowFrq}
\end{figure}

Finally we show an example where a lot of low frequencies are involved in the high dimensional regime :

$$
f(x) = 2 \pi^2 \sum_{k=1}^{d-1} \cos(\pi \cdot x_k)\cos(\pi \cdot x_{k+1})
$$
whose solution is :

$$
u^\star(x) = \sum_{k=1}^{d-1} \cos(\pi \cdot x_k)\cos(\pi \cdot x_{k+1}).
$$

For $d=6$, $m = 1000$ and all other parameters being identical to previous cases, one gets convergence of the solution on Figure~\ref{mixedMode} where the contour plot still shows the function's values on the slice $(x_1,x_2, 0.5, \cdots, 0.5)$.

\begin{figure}[H]
    \begin{subfigure}{.5\textwidth}
    	\centering
		\includegraphics[width=.8\linewidth]{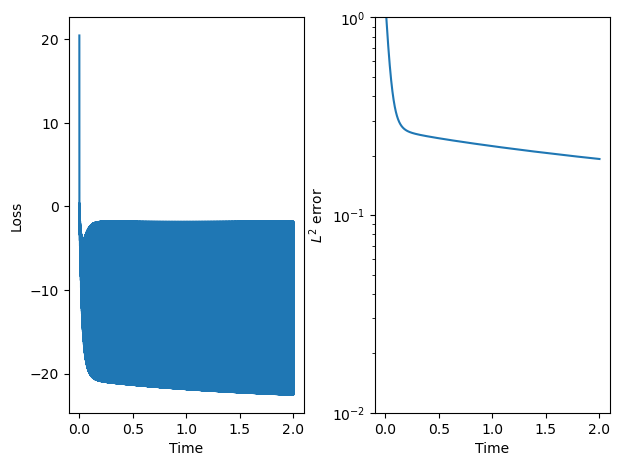}
		\caption{The loss and the $L^2$ error}
    \end{subfigure}
    \begin{subfigure}{.5\textwidth}
    	\centering
        \includegraphics[width=.8\linewidth]{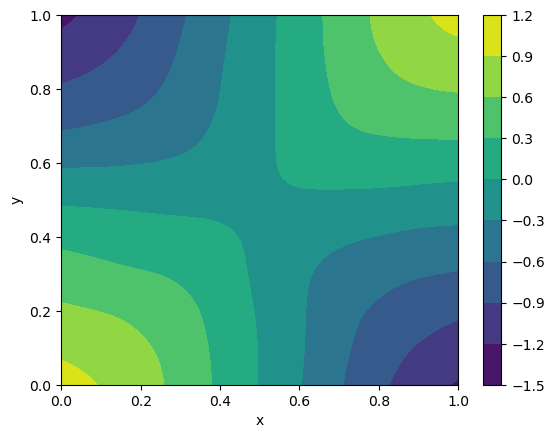}
        \caption{The numerical solution}
    \end{subfigure}
    \caption{The mixed mode solution}
    \label{mixedMode}
\end{figure}

\section{Conclusion}

In this article, the ability of two-layer neural networks to solve Poisson equation is investigated. First the PDE problem commonly understood in the Sobolev sense, is reinterpreted in the perspective of probability measures by writing the energy functional as a function over probabilities. Then, we propose to solve the obtained minimization problem thanks to gradient curves for which an existence result is shown. To justify this choice of method, the convergence towards an optimal measure is proved assuming the convergence of the gradient curve. Finally, numerical illustrations with a detailed analysis of the effects of dimension and frequency are presented. With this work, it becomes clear that neural networks is a viable method to solve Poisson equation even in the high dimensional regime; something out of reach for classical methods. Nonetheless, some questions and extensions deserve more detailed developments. First, the main remark to observe is that the convergence is not proved theoretically even if it is observed in practice. Additionally, the domain considered is very peculiar $\Omega = [0,1]^d$ and it is not obvious that one could generalize such theory on domain where sin/cosine decomposition is not available.
In numerical illustrations, integrals involved in the cost were not computed exactly but approximated by uniform sampling. It should be interesting to study the convergence of gradient curves with respect to the number of samples.

\appendix 
\section{The differential structure of Wasserstein spaces over compact Alexandrov spaces}\label{appendix}

The aim of this section is to get acquainted of the differential structure of $\Prob_2(\Theta)$. All the results presented here are not rigorously proved and we rather give a didactic introduction to the topic, the main reference being~\cite{Ohta2009}.

\subsection{The differential structure of Alexandrov spaces}

An Alexandrov space $(A,d)$ is a geodesic space embedded with its distance $d$ having a nice concave property on triangles. Roughly, Alexandrov spaces are spaces where the curvature is bounded from below by a uniform constant. Before going further, we need to introduce some notation :

\begin{definition}
Let $\alpha$ be a unit speed geodesic with $\alpha(0) = a \in A$ and $s\geq 0$, then we introduce the notation :
$$
(\alpha,s) : \mathbb{R}_+ \ni t \mapsto \alpha(st)
$$
the associated geodesic of velocity $s$. We then make the identification
$$
"(\alpha, 1) = \alpha"
$$
unit speed geodesic $\alpha$. 
\end{definition}

It is not so important to focus on a rigorous definition of such spaces but one should remember the following fundamental property of existence of a tangential cone structure :

\begin{theorem}
Let $\alpha, \beta$ be two unit speed geodesics with $\alpha(0) = \beta(0) =: a \in  A$ and $s,t\geq 0$. Then the limit :
$$
\sigma_a((\alpha,s), (\beta,t)) := \lim_{\varepsilon \rightarrow 0} \frac{1}{\varepsilon} d(\alpha(s\varepsilon), \beta(t\varepsilon))
$$
exists. Moreover,
\begin{equation}\label{cos}
\frac{1}{2st} \left( s^2 + t^2 -  \sigma_a((\alpha,s), (\beta,t))\right)
\end{equation}
depends neither on $s$ nor on $t$.
\end{theorem}

The previous theorem is very important as it enables to introduce a notion of angle and scalar product :

\begin{corollary}
One can define the local angle $\angle_a((\alpha,s), (\beta,t))$ between $(\alpha,s)$ and $(\beta,t)$ by :
$$
\cos(\angle_a((\alpha,s), (\beta,t))) := \frac{1}{2st} \left( s^2 + t^2 -  \sigma_a((\alpha,s), (\beta,t))\right)
$$
and a local scalar product :
$$
\langle (\alpha,s), (\beta,t) \rangle_a :=  st \cos(\angle_a((\alpha,s),(\beta,t))).
$$
\end{corollary}

We then have the following definitions. 
\begin{definition}
The space of directions $\Sigma_a(A)$ is the completion of 
$$
\left\{ (\alpha,1) \ | \ \alpha \text{ unit speed geodesic departing from a } \right\}
$$ quotiented by the relationship $\sigma_a = 0$ with respect to the distance $\sigma_a$.

The tangent cone, i.e. the set of geodesics departing from $a$ at speed $s$, of the form $(\alpha,s)$ for some $(\alpha,1)\in \Sigma_a(A)$, is denoted by $C_a(A)$.
\end{definition}

A major result from~\cite{Ohta2009} is that if the underlying space $A$ is Alexandrov and compact then the space over probabibilty $\Prob_2(A)$ is also an Alexandrov space and all the differential structure presented above is available. The proof of this result is based on McCann interpolation which allows to make the link between probability geodesics and geodesics of the underlying space.

Moreover, it is possible to define a notion of differentiation.

\begin{definition}
For a curve $(a_t)_{t\in \mathbb{R}}$ of $A$, it is said to be differentiable at $t = 0$ if there exists $(\alpha,\tau) \in C_a(A)$ such that for all $(\alpha_i,1) \in \Sigma_a(A)$, $t_i \geq 0$ with $\lim_{i \rightarrow \infty} t_i = 0$, linking $a_0$ and $a_{t_i}$ then :
$$
\mathop{\lim}_{i \rightarrow \infty} (\alpha_i, d(a_0,a_{t_i})/t_i) = (\alpha, \tau)
$$ 
where the convergence has to be understood in the sense of the distance $\sigma_a$. Moreover, the derivative of the curve at $t=0$ writes :
$$
a_0^\prime := (\alpha,\tau).
$$
\end{definition}

\subsection{The notion of gradient}

Now let us consider an energy $\E: A \rightarrow \R$ with the following property of convexity.

\begin{definition}\label{definitionConvexGeodesic}
We say that $\E$ is convex along geodesics if there exists $K \in \R$ such that for all rescaled geodesics $\alpha : [0,1] \rightarrow A$ :

$$
\E(\alpha(\lambda)) \leq (1-\lambda) \E(\alpha(0)) + \lambda \E(\alpha(1)) - \frac{K}{2} \lambda (1 - \lambda) d(\alpha(0), \alpha(1)).
$$
\end{definition}

Assuming such convexity, it is possible to define the gradient's direction of $\E$ using the differential structure of $A$ (see~\cite[Lemma 4.3]{Ohta2009}). Before doing this, it is necessary to introduce the directional derivative :

\begin{definition}
For $a \in A$ and $(\alpha,s) \in C_a(A)$, one defines :

$$
D_a\E((\alpha,s)) := \lim_{\varepsilon \rightarrow 0} \frac{\E(\alpha(s\varepsilon)) - \E(\alpha(0))}{\varepsilon}. 
$$
\end{definition}
One can prove that the limit above exists using the convexity assumption of $\E$. Owing this, there exists a direction for which the local slope (see Definition~\ref{definitionLocalSlope}) is attained in the sense defined below.
\begin{theorem}
For all $a \in A$ such that $|\nabla_- \E|(a) < \infty$, there exists a unique direction $(\alpha,1) \in \Sigma_a(A)$ such that :

$$
D_a \E((\alpha,1)) = - |\nabla_- \E|(a).
$$

This direction $\alpha$ is denoted by $
\frac{\nabla_- \E(a)}{|\nabla_- \E|(a)}
$, which means that :

$$
D_a \E((\alpha,|\nabla_- \E|(a))) := - |\nabla_- \E|^2(a).
$$
\end{theorem}

With this, it is straightforward to define the notion of gradient curve.

\begin{definition}
A Lipschitz curve $(a_t)_{t\geq 0}$ is said to be a gradient curve with respect to $\E$ if it is differentiable for all $t\geq 0$ and :

$$
\forall t \geq 0, \  a^\prime_t = \left(\frac{\nabla_- \E(a_t)}{|\nabla_- \E|(a_t)}, |\nabla_- \E|(a_t)\right) \in C_{a_t}(A).
$$

\end{definition}

In~\cite{Ohta2009}, results about existence and uniqueness of gradient curve on $\Prob_2(A)$ are given.

\section*{Acknowledgements}

The authors acknowledge funding from the Tremplin-ERC Starting ANR grant HighLEAP (ANR-22-ERCS-0012).

\newpage 

\bibliographystyle{unsrt}
\bibliography{sample}

\end{document}